\documentclass[onefignum,onetabnum]{siamart220329}

\usepackage{lipsum}
\usepackage{amsfonts}
\usepackage{graphicx}
\usepackage{epstopdf}
\usepackage{algorithm}
\usepackage{algpseudocode}
\ifpdf
  \DeclareGraphicsExtensions{.eps,.pdf,.png,.jpg}
\else
  \DeclareGraphicsExtensions{.eps}
\fi

\newsiamremark{remark}{Remark}
\newsiamremark{hypothesis}{Hypothesis}
\crefname{hypothesis}{Hypothesis}{Hypotheses}
\newsiamthm{claim}{Claim}

\headers{Convergence of Fisher-Rao Gradient Flow}{}

\title{Fisher-Rao Gradient Flow: Geodesic Convexity and Functional Inequalities
\thanks{Submitted to the editors DATE. The authors are in alphabetical order.
\funding{JAC was supported by the Advanced Grant Nonlocal-CPD (Nonlocal PDEs for Complex Particle Dynamics: Phase Transitions, Patterns and Synchronization) of the European Research Council Executive Agency (ERC) under the European Union’s Horizon 2020 research and innovation programme (grant agreement No.~883363). JAC was also partially supported by the EPSRC grants EP/T022132/1 and EP/V051121/1 and by the “Maria de Maeztu” Excellence Unit IMAG, reference CEX2020-001105-M, funded by MCIN/AEI/10.13039/501100011033/. YC acknowledged support from the Courant Instructorship. DZH was supported by National Natural Science Foundation of China through grant 12471403 and the Fundamental Research Funds for the Central Universities of China. }}}

\author{
José A. Carrillo\thanks{Mathematical Institute, University of Oxford, Oxford OX2 6GG, UK (\email{carrillo@maths.ox.ac.uk}).}
\and
Yifan Chen\thanks{Department of Mathematics, University of California, Los Angeles, CA, USA (\email{yifanchen@math.ucla.edu})}
\and 
Daniel Zhengyu Huang\thanks{Beijing International Center for Mathematical Research,  Center for Machine Learning Research, Peking University, Beijing, China (\email{huangdz@bicmr.pku.edu.cn}).}
\and
Jiaoyang Huang\thanks{University of Pennsylvania, Philadelphia, PA, USA (\email{huangjy@wharton.upenn.edu}).}
\and
Dongyi Wei\thanks{School of Mathematical Sciences, Peking University, Beijing, China (\email{jnwdyi@pku.edu.cn}).}
}

\usepackage{amsopn}

\makeatletter
\newcommand*{\addFileDependency}[1]{% argument=file name and extension
  \typeout{(#1)}% latexmk will find this if $recorder=0 (however, in that case, it will ignore #1 if it is a .aux or .pdf file etc and it exists! if it doesn't exist, it will appear in the list of dependents regardless)
  \@addtofilelist{#1}% if you want it to appear in \listfiles, not really necessary and latexmk doesn't use this
  \IfFileExists{#1}{}{\typeout{No file #1.}}% latexmk will find this message if #1 doesn't exist (yet)
}
\makeatother

\usepackage{appendix}
\usepackage{subfig}
\usepackage{graphicx}
\usepackage{color}
\usepackage{bm}
\usepackage{enumitem}
\usepackage{bbm}
\newcommand{\N}{\mathcal{N}}

\newcommand{\G}{\mathcal{G}}

\newcommand{\E}{\mathbb{E}}

\newcommand{\calH}{\mathcal{H}}

\newcommand{\calD}{\mathcal{D}}
\newcommand{\PP}{\mathcal{P}}

\newcommand{\EE}{\mathcal{E}}

\newcommand{\M}{M}

\newcommand{\R}{\mathbb{R}}

\newcommand{\bigO}{\mathcal{O}}

\newcommand{\dd}{\mathrm{d}}
\newcommand{\rhopo}{{\rho^{*}}}
\newcommand{\thetapo}{{\theta^{*}}}
\newcommand{\dualf}{{\bar{f}}}
\newtheorem{condition}{Condition}[section]

\newenvironment{newremark}[1]{%
    \begin{remark}#1}{%
    \Endofdef\end{remark}%
}
\newcommand{\xqed}[1]{%
    \leavevmode\unskip\penalty9999 \hbox{}\nobreak\hfill
    \quad\hbox{\ensuremath{#1}}}
\newcommand{\Endofdef}{\xqed{\lozenge}}

\definecolor{darkred}{rgb}{.6,0,0}
\definecolor{darkblue}{rgb}{0,0,.7}
\definecolor{darkgreen}{rgb}{0,.7,0}
\definecolor{darkbrown}{rgb}{0.8,0.4,0.4}

\begin{document}

\maketitle
% REQUIRED
\begin{abstract}
The dynamics of probability density functions have been extensively studied in computational science and engineering to understand physical phenomena and facilitate algorithmic design. Of particular interest are dynamics formulated as gradient flows of energy functionals under the Wasserstein metric. The development of functional inequalities, such as the log-Sobolev inequality, plays a pivotal role in analyzing the convergence of these dynamics.
This paper aims to extend the success of functional inequality techniques to dynamics that are gradient flows under the Fisher-Rao metric, with various $f$-divergences serving as energy functionals. Such dynamics take the form of nonlocal differential equations, for which existing analyses critically rely on explicit solution formulas in special cases.
We provide a comprehensive study of functional inequalities and the relevant geodesic convexity for Fisher-Rao gradient flows under minimal assumptions. A notable feature of our functional inequalities is their independence from the log-concavity or log-Sobolev constants of the target distribution. Consequently, the convergence rate of the dynamics (assuming well-posedness) remains uniform across general target distributions.
\end{abstract}

% REQUIRED

% REQUIRED

\tableofcontents

\section{Introduction}
The study of dynamics of probability density functions has proven invaluable across various disciplines, including statistical physics, developmental biology, and the design of efficient algorithms for uncertainty quantification. A particularly important application is sampling a target positive probability density $\rhopo: \mathbb{R}^{d} \to \mathbb{R}$ known up to normalization constants:
\begin{equation}
    \rhopo(\theta) \propto \exp(-V(\theta)),
\end{equation}
in terms of a potential $V(\theta)$. For example, such sampling problems arise in climate science for calibrating climate models~\cite{isaac2015scalable,schneider2017earth,huang2022iterated,lopez2022training}, engineering for decision making and risk management~\cite{yuen2010bayesian,cui2016dimension,ghattas2021learning,cao2022bayesian}, tissue growth and cell populations models in biological engineering~\cite{Hines2014DeterminationOP,simpson2020practical,FCB23}, and machine learning for quantifying prediction uncertainties~\cite{murphy2012machine,hernandez2015probabilistic,maddox2019simple,yang2021b,chen2021solving}.
The demand from these diverse applications has spurred extensive research into efficient sampling algorithms, a subject explored across applied mathematics, computer science, machine learning, and statistics.

Numerous approaches, based on constructing dynamics of probability density functions, have been proposed in the literature to address the sampling problem. The most widely used examples in Bayesian inference are sequential Monte Carlo (SMC) \cite{doucet2009tutorial} and Markov chain Monte Carlo (MCMC) \cite{brooks2011handbook}. A classical instance in the MCMC category (in continuous time) is the Langevin dynamics \cite{pavliotis2014stochastic}
\begin{equation}
    {\rm d}X_t = \nabla \log \rho^*(X_t) {\rm d}t + \sqrt{2}{\rm d}B_t,
\end{equation}
and the corresponding Fokker-Planck equation
\begin{equation}
\label{eqn-Wasserstein-gradient-flow}
    \frac{\partial \rho_t}{\partial t} + \nabla \cdot (\rho_t \nabla \log \rho^*) = \Delta \rho_t,
\end{equation}
which satisfies that when $t \to \infty$, the law of $X_t$, namely $\rho_t$, converges to $\rho^*$, under certain regularity conditions on $\rho^*$. The study of the convergence rate has been of significant interest, as it characterizes the mixing time of the dynamics and the efficiency of the sampling algorithm~\cite{jordan1998variational,chewi2023logconcave}.

A remarkable observation is that the Fokker-Planck equation can be formulated as the \textit{gradient flow} of the Kullback–Leibler (KL) divergence in information theory, known as the Boltzmann relative entropy in kinetic theory \cite{To}, with respect to the Wasserstein metric, according to the seminal work of Jordan, Kinderlehrer, and Otto \cite{jordan1998variational}. For this reason, \eqref{eqn-Wasserstein-gradient-flow} is referred to as the \textit{Wasserstein gradient flow} of the KL divergence. This gradient flow structure has profoundly influenced the analysis of the dynamics \cite{carrillo2001entropy,CMV03,CMV06,otto2000generalization} and the design of new sampling algorithms \cite{wibisono2018sampling,bernton2018langevin, chen2023sampling}. 

The goal of this paper is to systematically extend the success of the analysis to \textit{Fisher-Rao gradient flows} which are nonlocal differential equations of great interest in statistics, information geometry and evolutionary game theory. These flows have the potential to achieve fast, universal convergence rates for any target distribution, in contrast to Wasserstein gradient flows. In \cref{subsec-Wasserstein gradient flows}, we review the results for Wasserstein gradient flows, including functional inequalities (e.g., the log-Sobolev inequality), displacement convexity, and their induced convergence rates of the dynamics. \cref{subsec-Fisher-Rao gradient flows} introduces Fisher-Rao gradient flows,  explaining their significance in sampling applications and the existing methods used to analyze their convergence. \cref{subsec-Main results} presents a summary of our main results for a comprehensive study of geodesic convexity and functional inequalities in the Fisher-Rao setting. We demonstrate the favorable, uniform convergence rates of Fisher-Rao gradient flows across general target distributions. Related works are discussed in \cref{subsec-Related works}, and the organization of this paper is outlined in \cref{subsec-Organization}.

\subsection{Wasserstein gradient flows: Functional inequalities and convexity}
\label{subsec-Wasserstein gradient flows}
For the scope of this paper, we adopt the formal approach by Otto \cite{otto2001geometry} to define gradient flows of probability densities. To do so, consider the probability space with smooth positive densities:
\begin{equation}
\label{eqn-smooth-positive-densities}
\mathcal{P} = \Bigl\{\rho \in C^{\infty}(\R^{d}) : \int \rho(\theta) \mathrm{d}\theta = 1 ,\, \rho > 0\Bigr\}\, .
\end{equation}
The Riemannian metric at $\rho$ is denoted by $g_{\rho}$ such that
\begin{equation}
    g_{\rho}(\sigma_1,\sigma_2) = \langle M(\rho)\sigma_1, \sigma_2\rangle \ \text{for}\ \sigma_1, \sigma_2 \in T_{\rho}\PP,
\end{equation}
where $\langle\cdot,\cdot\rangle$ is the duality pairing between the cotangent space $T^*_{\rho}\PP$ and tangent space $T_{\rho}\PP$, and can be identified as the $L^2$ inner product given both $\sigma_1, \sigma_2$ are classical functions. Here, $M(\rho): T_{\rho}\PP \to T_{\rho}^*\PP$ is a metric tensor. Given an energy functional $\mathcal{E}: \mathcal{P} \to \mathbb{R}$, the gradient flow under the metric tensor $M(\rho)$ is
\begin{align}
\label{eq:gf}
\frac{\partial \rho_t}{\partial t}
= -{\rm grad}_M \EE(\rho_t) = -\M(\rho_t)^{-1} \frac{\delta \EE}{\delta \rho}\Bigr|_{\rho_t},
\end{align}
where $\frac{\delta \EE}{\delta \rho}$ is the first variation of the energy functional $\EE$ such that
\begin{align}
\label{eqn-first-variation-def}
\Bigl\langle \frac{\delta \mathcal{E}}{\delta \rho}, \sigma \Bigr\rangle
=  \lim_{\epsilon \rightarrow 0} \frac{\mathcal{E}(\rho + \epsilon \sigma) - \mathcal{E}(\rho)}{\epsilon},
\end{align}
and we require $\mathbb{E}_{\rho}[\frac{\delta \mathcal{E}}{\delta \rho}] = 0$ so that the first variation is unique.

By setting the energy function as the KL divergence 
\begin{equation}
\label{eq:KL}
    \EE(\rho) = \text{KL}[\rho\Vert \rho^*] = \int \rho \log \frac{\rho}{\rho^*}{\rm d}\theta,
\end{equation} and using the Wasserstein metric tensor $M(\rho)^{-1}\psi = -\nabla \cdot (\rho \nabla  \psi)$, we obtain the Wasserstein gradient flow of the KL divergence:
\begin{equation}
\label{eq:W-gf}
    \frac{\partial \rho_t}{\partial t}  = -{\rm grad}_M  \EE(\rho) = \nabla\cdot \Bigl(\rho_t \nabla\log\frac{\rho_t}{\rhopo}\Bigr),
\end{equation}
which attains the same form as the Fokker-Planck equation \eqref{eqn-Wasserstein-gradient-flow}.

The gradient flow structure of the Fokker-Planck equation naturally leads to Riemannian optimization interpretation for understanding its convergence \cite{otto2000generalization}. In fact, along the gradient flow, it holds that
\begin{equation}
    \frac{\dd}{\dd t} (\EE(\rho_t) - \EE(\rho^*))  =  - g_{\rho_t}(\textrm{grad}_M\EE(\rho_t), \textrm{grad}_M\EE(\rho_t)).
\end{equation}
If we have the following \textbf{gradient dominance} condition (also known as the Polyak-Łojasiewicz inequality~\cite{polyak1963gradient}):
\begin{equation}
\label{eq:intro-GDC-functional}
 \exists \alpha >0, \ g_{\rho}(\textrm{grad}_M\EE(\rho), \textrm{grad}_M\EE(\rho))
    \geq 2\alpha ( \EE(\rho) - \EE(\rhopo) ), \textrm{ for any } \rho \in \mathcal{P},
\end{equation}
then, it is straightforward to derive $\EE(\rho_t) - \EE(\rhopo) \leq  \exp(-2\alpha t)\bigl(\EE(\rho_0) - \EE(\rhopo)\bigr)$. In the case of $\EE(\rho) = \text{KL}[\rho\Vert \rho^*]$ and  $M(\rho)^{-1}\psi = -\nabla \cdot (\rho \nabla  \psi)$, the gradient dominance condition is equivalent to
\begin{equation}
\label{eq:intro-logSoblev.0}
\exists \alpha >0 \textrm{ s.t. }
    \int \rho \bigl|\nabla\log\frac{\rho}{\rhopo}\bigr|^2 \dd\theta \geq 2\alpha \int \rho\log\frac{\rho}{\rhopo} \dd\theta, \textrm{ for any } \rho \in \mathcal{P}.
\end{equation}
This is precisely the log-Sobolev inequality~\cite{gross1975logarithmic,gross1975hypercontractivity} for $\rhopo$ and $f = \sqrt{\frac{\rho}{\rhopo}}$:
\begin{equation}
\label{eq:intro-logSoblev}
    \int \bigl|\nabla f(\theta)\bigr|^2\rhopo(\theta)\dd\theta \geq \frac{\alpha}{2} \int f(\theta)^2\log\Bigl(\frac{f(\theta)^2}{\int f(\theta')^2 \rhopo(\theta')\dd\theta'}\Bigr) \rhopo(\theta)\dd\theta.
\end{equation}
A sufficient (although not necessary) condition for the gradient dominance condition is the strong \textbf{geodesic convexity}:
\begin{align}
    \label{eq:alpha-uniformly geodesic convex}
         \exists \alpha >0,\ g_{\rho}(\textrm{Hess}_M \EE(\rho) \sigma, \sigma) \geq \alpha g_{\rho}(\sigma, \sigma),\textrm{ for any } \rho \in \mathcal{P}, \sigma \in T_{\rho}\PP,
    \end{align} where $\textrm{Hess}_M \EE(\rho)$ is the Riemannian Hessian of $\EE$. For the Wasserstein case where $M(\rho)^{-1}\psi = -\nabla \cdot (\rho \nabla  \psi)$, the geodesic convexity is also known as the displacement convexity \cite{mccann1997convexity}. The Bakry-Emery theorem~\cite{bakry1985diffusions, bakry2014analysis} establishes the connection between the displacement convexity, the log-concavity of $\rho^\star$, and the Ricci curvature of the sample space, and conditions for general convex Sobolev inequalities involving different relative entropies or divergences were established in \cite{To,AMTU}. In summary, these functional inequalities and convexity properties play important roles for studying the convergence of dynamics arising from Wasserstein gradient flows \cite{mccann1997convexity,CMV03,ambrosio2005gradient}.

We note that while the above explanations regarding the geometry and gradient flow structures are mostly formal, they can be made rigorous \cite{ambrosio2005gradient,Santambook}. 
On the other hand, it is also quite important to strategically use the geometric structure to formally derive the desired inequalities and then prove them rigorously. 
Once well-posedness of the dynamics has been established, these inequalities can be used to establish their convergence.
In this work, we focus on rigorously proving the functional inequalities.

\subsection{Analysis of Fisher-Rao gradient flows}
\label{subsec-Fisher-Rao gradient flows}
In this paper, our main focus is on the gradient flow dynamics with respect to the Fisher-Rao metric \cite{rao1945information,amari2016information,ay2017information,laschos2019geometric}, which takes the following form
\begin{equation}
\label{eq:FR-metric}
    g_{\rho}^{\mathrm{FR}}(\sigma_1,\sigma_2)=\int \frac{\sigma_1\sigma_2}{\rho} \dd\theta, \text{ for } \sigma_1,\sigma_2 \in T_{\rho}\PP.
\end{equation}
The corresponding metric tensor is $M^{\mathrm{FR}}(\rho)^{-1}\psi = \psi - \mathbb{E}_{\rho}[\psi]$. Historically, this metric was originally referred to as the spherical Hellinger metric, since it arises as the restriction of the Hellinger metric from the space of positive measures to the space of probability measures~\cite{hellinger1909neue}. More recently, it has become widely known as the Fisher–Rao metric, which is the terminology we adopt in this work; see \cref{remark:Fisher-Rao-vs-Hellinger} for details.

As an example, the Fisher-Rao gradient flow of the KL-divergence can be expressed as follows:
\begin{equation}
\label{eq:FR-gf}
    \frac{\partial \rho_t}{\partial t} = -{\rm grad}_{\rm FR} {\rm KL}[\rho_t\Vert\rhopo] = -\rho_t \Bigl(
 \log \rho_t - \log \rhopo - \E_{\rho_t}[ \log \rho_t -
\log \rhopo] \Bigr).
\end{equation}
Although this Fisher-Rao gradient flow takes the form of an ODE at each $\theta$, it contains a \textit{nonlocal} interaction term $\mathbb{E}_{\rho_t}[\log \rho_t - \log \rho^*]$ that conserves the mass. It is nontrivial to analyze such nonlocal terms. 
Recently, convergence of the above Fisher-Rao gradient flow \eqref{eq:FR-gf}, namely
\begin{equation}
\label{eq:KL-FR-GF-convergence}
    {\rm KL}[\rho_t\Vert \rhopo] = \bigO(\exp(-t)),
\end{equation}
has been established in \cite{lu2019accelerating,chen2023sampling,lu2022birth,domingo2023explicit}. We highlight that the exponential convergence rate here is uniform and independent of $\rho^*$. This characteristic stands in sharp contrast to the Wasserstein gradient flow~\eqref{eq:W-gf}, where the convergence rate $\bigO(\exp(-\alpha t))$ depends on the log-Sobolev constant $\alpha$; this constant depends on $\rho^*$ and can be very small for anisotropic or multimodal distributions. As a consequence, the uniform convergence rate of the Fisher-Rao gradient flow makes it a potentially desirable dynamics for sampling applications; see \cite{lu2019accelerating,lu2022birth,huang2022efficient,lindsey2022ensemble,chen2024efficient,maurais2024sampling}. 

Nevertheless, most existing proofs of the uniform convergence of the Fisher-Rao gradient flow~\eqref{eq:FR-gf} rely on the explicit formula of the solution
\begin{equation}
    \rho_t(\theta) \propto \rho_0(\theta)^{e^{-t}}\rhopo(\theta)^{1 - e^{-t}},
\end{equation}
which is specific to the KL divergence as the energy functional. Moreover, the convergence results require strong assumptions on $\rho_0$. This prevents us from extending the convergence framework of Fisher-Rao gradient flows to general energy functionals and initial conditions $\rho_0$.

% We note that there has also been work using convexity to study the convergence of Fisher-Rao gradient flows. For a general energy functional $\EE(\rho)$, the flow has the form
% \begin{equation}
% \label{eq:FR-gf-E}
%     \frac{\partial \rho_t}{\partial t} = -{\rm grad}_{\rm FR} \EE[\rho_t] = -\rho_t \Bigl(\frac{\delta \EE}{\delta \rho}\Bigl|_{\rho_t}  - \E_{\rho_t}\bigl[ \frac{\delta \EE}{\delta \rho}\Bigl|_{\rho_t}\bigr]\Bigr).
% \end{equation}
% The equation can be interpreted as the continuous-time limit of mirror descent~\cite{beck2003mirror} in the probability density space~\cite{aubin2022mirror}, where the Bregman divergence is chosen as the KL divergence. 

% Techniques in convex analysis have be used to study the convergence of the flow.
% The associated Bregman generator is the negative entropy functional, $\mathcal{H}:\rho \rightarrow \int \rho \log \rho \dd\theta$, and the mirror descent update takes the form:
% \begin{equation}
%     \label{eq:mirror-descent}
%     \rho_{ n+1} = \argmin_{\rho} \Bigl\{\,\EE(\rho) + \frac{1}{\Delta t}{\rm KL}[\rho\Vert\rho_n]\,\Bigr\}.
% \end{equation}
% More specifically, 
We note that there has also been work using convexity to study the convergence of Fisher-Rao gradient flows. When the energy functional $\EE(\rho)$ is  
strongly relatively convex (in the sense of \cite{lu2018relatively}) with respect to the KL divergence, namely
\begin{equation}
\label{eq:rel-convex}
   \EE(\rho') \geq \EE(\rho) + \int \frac{\delta \EE}{\delta \rho}\Bigr|_{\rho}(\rho' - \rho )\dd\theta  + \alpha \rm{KL}[\rho'\Vert\rho],  \quad \forall \rho,\,\rho' ,
\end{equation}
for some $\alpha>0$, the Fisher-Rao gradient flow can be shown to converge exponentially fast \cite[Theorem 1]{yao2024minimizing} in reverse KL divergence: ${\rm KL}[\rhopo \Vert \rho_t] \leq e^{-\alpha t}{\rm KL}[\rhopo \Vert \rho_0]$. In particular, when the energy functional is the KL divergence itself, the strong relative convexity condition \eqref{eq:rel-convex} holds with $\alpha = 1$. Thus, convergence of ${\rm KL}[\rhopo \Vert \rho_t]$ for \eqref{eq:FR-gf} is guaranteed; however, convergence in the original KL divergence ${\rm KL}[\rho_t \Vert \rhopo]$ is not covered in this framework. We note that convergence of the discrete iteration of the Fisher-Rao gradient flow and its connection to mirror descent has also been studied in \cite{chopin2023connection} based on techniques in \cite{aubin2022mirror}.

This paper aims to provide a more systematic study of convergence for the Fisher-Rao gradient flows with general $f$-divergences as energy functionals, based on functional inequalities and geodesic convexity that have proven successful for studying Wasserstein gradient flows. Here,
the $f$-divergences are defined as \begin{align}
D_f[\rho \Vert  \rhopo]  =\int\rhopo f\Bigl(\frac{ \rho}{ \rhopo}\Bigr)\,\dd\theta,
\end{align}
where $f(1)=0, f\in C^2(0, +\infty)$ and $f$ is convex. It is worthy noting that $f$-divergences are known as relative entropies in the context of statistical physics and kinetic theory, see for instance \cite{To,AMTU,HYKE}.

\subsection{Main results}
\label{subsec-Main results}
Firstly, we prove that the gradient dominance condition and the geodesic convexity do \textit{not} hold for the Fisher-Rao gradient flows of the KL divergence \eqref{eq:FR-gf}. This poses significant challenges for establishing the convergence of the dynamics using the functional inequality approach. The informal statement is presented below and the detailed analysis is in \cref{theorem:FR-displacement-convex,theorem:FR-gradient-dominant-condition}. Related results have been studied in the concurrent work \cite[Lemma 6.2]{ZhuMielke2024} which focuses on the setting without the mass preservation constraint.
\begin{theorem}[informal]
    Under the Fisher-Rao geometry, the KL divergence is not geodesically convex for any $\rhopo$, even within a small neighborhood of $\rhopo$. Moreover, the gradient dominance condition is not satisfied for any $\rhopo$, also within any small neighborhood of $\rhopo$. Here the ``small neighborhood'' is measured in the KL divergence.
\end{theorem}

Secondly, for general $f$-divergences as the energy functionals, we derive a sufficient condition on $f$ for geodesic convexity under the Fisher-Rao geometry, and a \textit{sufficient and necessary} condition \eqref{eq:intro_ns} on $f$ for the gradient dominance (see details in \cref{cond:gradient-dominant-condition-f}) to hold. These lead to general convergence results for the corresponding Fisher-Rao gradient flows.
We state our informal result below and the detailed result is in \cref{theorem:postive-uniform-convex,theorem:FR-gradient-dominant-condition-sn,theorem:FR-gradient-dominant-condition-suff}. We note that in contemporary work, \cite[Section 4]{mielke2025hellinger} studies the gradient dominance condition for a particular class of $f$-divergences with $f''(x) = x^p$.
\begin{theorem}[informal]
    Under the Fisher-Rao geometry, the $f$-divergence is geodesically convex for any $\rhopo$ when $xf'(x)$ is concave. Additionally under the condition $f''(1) > 0$, the $f$-divergence achieves $\alpha_f$-strong geodesic convexity for any $\rhopo$ with $\alpha_f > 0$ depending solely on $f$.

Moreover, for general $f$-divergences, the gradient dominance condition holds for any $\rhopo$ with an $\alpha_f > 0$ depending solely on $f$ if and only if
\begin{equation}
\label{eq:intro_ns}
\begin{split}
\exists \alpha > 0 \qquad 
 \Bigl(f'(y) - f'(x)\Bigr)^2  \geq  \alpha\left(\frac{1}{x} - \frac{1}{y}\right)\left( \frac{f(x)}{1 - x} - \frac{f(y)}{1 - y}\right)\quad  
\end{split}
\end{equation}
for all $0 < x < 1 < y$.
An sufficient condition is that there exists a constant $\alpha_s > 0$ such that $x^2f''(x) > \alpha_s$ for $0 < x \leq 1$. Note that the KL divergence does not satisfy the conditions.
\end{theorem}
The form of the sufficient and necessary condition may seem hard to parse at first glance. It can be derived from a special case of gradient dominance when $\rho$ and $\rhopo$ are two-point discrete densities; for more details see the proof for \cref{theorem:FR-gradient-dominant-condition-sn}.

Thirdly, we study the case of the KL divergence in further depth and seek other possible functional inequalities to establish the convergence, thereby extending previous results.
We prove a novel \textit{dual} type of gradient dominance conditions that allow us to obtain the convergence of the Fisher-Rao gradient flows for general $f$-divergences. To the best of our knowledge, such dual type of functional inequalities is not known in the Wasserstein case and it is a new property found in the Fisher-Rao case.

The informal statement is as follows; the detailed statement is in \cref{sec-Functional Inequality: Dual Gradient Dominance Condition}.  Related work \cite[Theorem 1]{yao2024minimizing} establishes that the KL divergence is a relatively strongly convex functional (in the sense of \cite{lu2018relatively}) with respect to the KL divergence itself, which leads to the convergence of the gradient flow under the reverse KL divergence, a case that can be covered by our result; see details in \cref{theorem:dual-ineq}.
\begin{theorem}[informal]
    For the Fisher-Rao gradient flow $\rho_t$ of an $f$-divergence, we consider its evolution under another $\bar{f}$-divergence. We define the dual gradient dominance condition on this flow as
    \begin{equation}
        \frac{\rm d}{{\rm d}t} D_\dualf[\rho_t\Vert \rhopo]\leq -\alpha_f \Bigl( D_\dualf[\rho_t \Vert \rhopo]  + D_f[\rho_t \Vert \rhopo] \Bigr).
    \end{equation}
    For $f\in C^2(0,+\infty)$ and $f''(1) > 0$, the above inequality holds with a constant $\alpha_f$ depending only on $f$ and independent of $\rho$ and $\rho^*$, if we take $\dualf(x) = \frac{(x-1)^2}{x}$. In this case, the dual divergence $D_{\dualf}$ corresponds to the reverse $\chi^2$ divergence. 

    Furthermore, for the KL divergence, if we take $\dualf(x) = xf(\frac{1}{x})$, which corresponds to the reverse KL divergence, the dual gradient dominance condition holds with an explicit $\alpha_f =1$. This can also be generalized to $f$-divergences with $f$ being a polynomial.
\end{theorem}

With these established functional inequalities, it is straightforward to show the corresponding gradient flows converge to the target exponentially fast at the same rate, given that the gradient flows are well posed, allowing us to differentiate the energy functional along the dynamics. The main focus of this paper is the rigorous proof of the functional inequalities under appropriate sufficient and/or necessary conditions, while we leave the study of the well-posedness of the gradient flows for future work. 

\subsection{Related works} 
\label{subsec-Related works}
\subsubsection{Displacement convexity and functional inequalities} 
With respect to the Wasserstein metric, geodesic convexity is equivalent to the displacement convexity, a concept introduced by McCann in \cite{mccann1997convexity}.
Further exploration of displacement convexity on Riemannian manifolds and more general spaces can be found in \cite{otto2000generalization,cordero2001riemannian,otto2005eulerian,von2005transport,lott2009ricci,ohta2011displacement}.
Additionally, the work \cite{carrillo2010nonlinear} studies the displacement convexity concerning the generalized Wasserstein metric with nonlinear mobility. These concepts have been extended to more general settings such as graphs, discrete space, Markov chains with other metrics, see for instance \cite{EM14,EMM19,EFMM22}.

Various functional inequalities related to the Wasserstein metric have been developed, such as the Poincar\'e inequality, log-Sobolev inequality, and Talagrand's inequality; see \cite{poincare1890equations, talagrand1996transportation,otto2000generalization}. These functional inequalities can be used to study the convergence of Wasserstein gradient flows and related PDEs \cite{otto2000generalization}. Log-Sobolev type inequalities have been proved for general divergences or relative entropy and for general nonlinear and nonlocal McKean-Vlasov equations; see for instance \cite{AMTU,carrillo2001entropy,CMV03,ambrosio2005gradient}.
However, few functional inequalities have been developed for the Fisher-Rao gradient flow.

\subsubsection{Fisher-Rao gradient flow}

The Fisher-Rao metric is widely used in finite-dimensional parametric density spaces, where it reduces to the Fisher information matrix, a fundamental concept in information geometry~\cite{amari2016information,ay2017information}; see also \cite{rao1945information,FRmetricInfDim1991,srivastava2007riemannian}.
In probability density spaces, Fisher-Rao and Wasserstein metrics have been studied in combination, leading to the Wasserstein-Fisher-Rao (also known as Hellinger–Kantorovich) metric~\cite{liero2018optimal,kondratyev2016new,chizat2018interpolating,chizat2018unbalanced}. A key advantage of this framework is its ability to handle unnormalized probability densities that are not constrained to integrate to one, giving rise to unbalanced optimal transport, which has found applications across various domains.
The interpolation can also be defined via a generalization of the dynamical Benamou-Brenier formulation~\cite{benamou2000computational} by incorporating a source term measured by the Fisher-Rao metric in the continuity equation. The associated gradient flows~\cite{kondratyev2019spherical} and functional inequalities~\cite{kondratyev2020convex,mielke2025hellinger} have been studied following Otto's formalism. In the present work, we focus on the Fisher–Rao metric for normalized probability densities and the associated gradient flows preserve the probability mass.

The Fisher-Rao gradient flow of the KL divergence \eqref{eq:FR-gf} takes the form
\begin{equation*}
\frac{\partial \rho_t}{\partial t} = -\rho_t \Lambda_t, \qquad  \Lambda_t  =
\bigl( \log \rho_t - \log \rhopo\bigr) - \E_{\rho_t}[ \log \rho_t -
\log \rhopo].
\end{equation*}
It is a mean-field equation of birth-death type \cite{lu2019accelerating,lu2022birth}
where the birth-death rate $\Lambda_t$ contains a nonlocal interaction term that allows global movement of mass. Similar equations in discrete space, known as the replicator equations, have been studied in evolutionary game theory \cite{hofbauer2003evolutionary}. Projection of the Fisher-Rao gradient flows to a parametric family of distributions leads to natural gradient flows. Time discretizations of these flows, known as natural gradient descent algorithms, have been widely used in variational inference \cite{amari1998natural,lin2019fast,martens2020new,chen2023sampling}.

Fisher-Rao gradient flows have found applications in various fields.  For example, the corresponding birth-death process has been used in sequential Monte Carlo samplers to reduce the variance of particle weights \cite{del2006sequential}, and in the acceleration of Langevin sampling \cite{lu2019accelerating,lu2022birth} and training of neural networks \cite{rotskoff2019global}. The Fisher-Rao gradient flow of the $\chi^2$ divergence arises as a mean-field equation for ensemble MCMC methods that address multimodal distributions \cite{lindsey2022ensemble}. 
Kernel approximations of Fisher-Rao gradient flows have been proposed for sampling \cite{maurais2024sampling,zhu2024kernel}. Furthermore, parametric approximation of Fisher-Rao gradient flows employing Kalman's methodology \cite{kalman1960new,evensen1994sequential,julier1997new} has led to the development of efficient, derivative-free posterior approximation algorithms. These algorithms, based on Gaussian or Gaussian mixture approximations \cite{huang2022efficient,chen2024efficient,che2025stable}, have proven particularly useful for Bayesian inverse problems. Other applications in filtering \cite{halder2018gradient} and reinforcement learning \cite{kerimkulov2023fisher} have also been explored in the literature.

\subsection{Organization of this paper}
\label{subsec-Organization}
In \Cref{sec-Preliminaries: Gradient Flows of Probability Density Functions}, we introduce Fisher-Rao gradient flows associated with various $f$-divergence.
In \Cref{sec-Geodesic Convexity}, we present an overview of convex analysis and study the geodesic convexity of  the $f$-divergence in the Fisher-Rao geometry.
In \Cref{sec-Functional Inequality: Gradient Dominance Condition}, we examine the gradient dominance condition of  $f$-divergence under the Fisher-Rao metric.
In \Cref{sec-Functional Inequality: Dual Gradient Dominance Condition}, we develop the dual gradient dominance condition for analyzing the convergence properties of $f$-divergences that do not satisfy the gradient dominance condition.
We conclude in \Cref{sec-Conclusion}.

\section{Fisher-Rao Gradient Flows of $f$-divergences}
\label{sec-Preliminaries: Gradient Flows of Probability Density Functions}
% \subsection{Fisher-Rao gradient flows of $f$-divergences}
We follow the approach in \cref{subsec-Wasserstein gradient flows} to derive the equation of Fisher-Rao gradient flows of $f$-divergences. Again, the derivation is formal, and our main goal is to analyze the relevant functional inequalities arising from such gradient flows, which allow us to establish convergence rates of the dynamics when they are well-posed.

We consider the energy functional to be the $f$-divergence
\begin{align}
\label{eqn-f-divergences}
\mathcal{E}(\rho) = D_f[\rho \Vert  \rhopo]  =\int\rhopo f\Bigl(\frac{ \rho}{ \rhopo}\Bigr)\,\dd\theta,
\end{align}
where $f(1)=0, f\in C^2(0, +\infty)$ and $f$ is convex. Here, the convexity of $f$ implies $$f(\rho/\rhopo)\geq f(1)+f'(1)(\rho/\rhopo-1),$$
which leads to
\begin{align}
\label{eq:D_f_lbound}
    D_f[\rho \Vert  \rhopo]  \geq \int\rhopo \Bigl(f(1)+f'(1)(\rho/\rhopo-1)\Bigr)\,\dd\theta \geq -\int \bigl| f'(1)( \rho - \rhopo)\bigr| \,\dd\theta  \geq -2\lvert f'(1)\rvert.
\end{align}
This means $f$-divergence is always well defined (with a possible value  $+\infty$).
Moreover, Jensen's inequality implies that $\rhopo$ is the unique minimizer. The KL-divergence
is a special $f$-divergence with $f = x\log x$.

For the $f$-divergence, its first variation with respect to $\rho$ is $f'\bigl(\frac{\rho}{\rhopo}\bigr) - \E_{\rho}[ f'(\frac{\rho}{\rhopo})]$ where we subtract a constant so that its mean is zero. Using the formula of the Fisher-Rao metric \eqref{eq:FR-metric}, we get the corresponding Fisher-Rao gradient flow as
\begin{align}
\label{eq:FR-gf-f}
\frac{\partial \rho_t}{\partial t}
= -\M(\rho_t)^{-1} \frac{\delta \EE}{\delta \rho}\Bigr|_{\rho_t} =  -\rho_t f'\Bigl(\frac{\rho_t}{\rhopo}\Bigr) + \rho_t \E_{\rho_t}\Bigl[ f'\Bigl(\frac{\rho_t}{\rhopo}\Bigr) \Bigr].
\end{align}
We note that the Fisher-Rao metric \eqref{eq:FR-metric} stands out as the unique metric, up to constants, that satisfies the diffeomorphism invariance property \cite{cencov2000statistical, ay2015information, bauer2016uniqueness}.
Under the Fisher-Rao metric, these gradient flows are expected to maintain this invariance.
The diffeomorphism invariance implies that the convergence property of Fisher-Rao gradient flows for a general target density $\rhopo$ aligns with that for the standard Gaussian target density, and may achieve a convergence rate independent of $\rhopo$. 

% For this reason, they have been widely used in sampling applications.

When $f(x)=x\log x$ which corresponds to the KL divergence, we get
\begin{equation}
    \frac{\partial \rho_t}{\partial t} = -\rho_t \Bigl(
 \log \rho_t - \log \rhopo - \E_{\rho_t}[ \log \rho_t -
\log \rhopo] \Bigr),
\end{equation}
which is the most commonly used Fisher-Rao gradient flow for sampling applications \cite{lu2019accelerating,chen2023sampling,lu2022birth,domingo2023explicit,yan2023learning}.
When $f(x) = (x-1)^2$ which corresponds to the $\chi^2$ divergence, we get
\begin{equation}
\label{eqn-chi-square-gradient-flow}
    \frac{\partial \rho_t}{\partial t} = -\rho_t \Bigl( \frac{\rho_t}{\rho^*}-\mathbb{E}_{\rho_t}[\frac{\rho_t}{\rho^*}]\Bigr),
\end{equation}
which has been used in the context of ensemble samplers \cite{lindsey2022ensemble}. 

Note that in \eqref{eqn-chi-square-gradient-flow}, evaluating the right-hand side appears to require knowledge of the normalization constant of $\rho^*$, which is typically unavailable in sampling applications. In fact, it has been demonstrated in \cite{chen2023sampling} that the KL divergence is unique among $f$-divergences, up to constants, in that gradient flows resulting from it do not depend on the normalization constant of the target distribution. Nevertheless, in \eqref{eqn-chi-square-gradient-flow}, the normalization constant appears as a multiplicative factor, and we can absorb it into the time variable and rescale the dynamics so that the right-hand side no longer depends on this constant; this approach changes the time scale. In \cite{lindsey2022ensemble}, ensemble approximation is employed with birth-death type updating rules to a slight variant of \eqref{eqn-chi-square-gradient-flow}, which does not require the normalization constant to implement.

\section{Geodesic Convexity in the Fisher-Rao Geometry}
\label{sec-Geodesic Convexity}
In this section, we explore the geodesic convexity of $f$-divergences in the Fisher-Rao geometry. We will present counterexamples to demonstrate that the KL divergence lacks geodesic convexity. Subsequently, we will provide sufficient conditions to identify the specific $f$-divergences that achieve geodesic convexity. The discussions in this section parallel those of displacement convexity in the Wasserstein geometry. 
The geometric calculations are primarily carried out formally, assuming the well-posedness of the geodesics and the gradient flow.
%\footnote{However, we note that they can be made rigorous by restricting to probability simplex $\sum_{i=1}^K \rho_i = 1, \rho_i > 0$, where the Riemannian geometry is rigorously defined.}. 
We first outline the preliminary concepts in the Euclidean space and abstract manifolds and then move to the Fisher-Rao geometry.
\subsection{Preliminaries on convex analysis} We start with introducing the basic concepts in the Euclidean space. Let $h$ be a convex function in $\mathbb{R}^N$, then for any $\theta_0, \theta_1 \in \mathbb{R}^N$ and $\lambda \in [0,1]$, it holds that
\begin{equation}
    h\bigl(\lambda\theta_0 + (1-\lambda)\theta_1\bigr) \leq \lambda h(\theta_0) + (1-\lambda)h(\theta_1).
\end{equation}
Furthermore suppose $h \in C^2(\mathbb{R}^N)$, then we have the first order characterization that
\begin{equation}
    h(\theta_1) \geq h(\theta_0) + \langle\nabla_\theta h(\theta_0),\theta_1 - \theta_0\rangle,
\end{equation}
and the second order characterization that the Hessian matrix $\nabla^2_{\theta} h(\theta) \succeq \alpha I$ for some $\alpha \geq 0$ and any $\theta \in \mathbb{R}^N$. When $\alpha > 0$, we call the function $h$ to be $\alpha$-strongly convex; in such case we have
\begin{equation}
\label{eq:alpha-convex-h}
    h(\theta_1) \geq h(\theta_0) + \langle\nabla_\theta h(\theta_0),\theta_1 - \theta_0\rangle + \frac{\alpha}{2}|\theta_1-\theta_0|^2.
\end{equation}
For a gradient flow in the Euclidean space, denoted by
$\frac{{\rm d}\theta_t}{{\rm d}t} = -\nabla_{\theta} h(\theta_t)$,
the evolution of the function values satisfy
$\frac{{\rm d}}{{\rm d}t} h(\theta_t) = -|\nabla_\theta h(\theta_t)|^2$.
\Cref{eq:alpha-convex-h} indicates that $h$ has a unique global minimizer, denoted by 
$\theta^*$. Using the strong convexity of $h$ we have
\begin{equation}
\label{eq:intro-GDC-f}
    h(\theta) - h(\thetapo) \leq \langle\nabla_\theta h(\theta), \theta - \thetapo \rangle - \frac{\alpha}{2} |\thetapo - \theta |^2 \leq \frac{1}{2\alpha} |\nabla_\theta h(\theta)|^2,
\end{equation}
where in the last inequality we used the fact that $\langle a, b\rangle\leq \frac{1}{2\alpha} |a|^2 + \frac{\alpha}{2}|b|^2$. Therefore, we obtain
\begin{equation}
    |\nabla_\theta h(\theta)|^2 \geq 2\alpha \bigl(h(\theta) - h(\thetapo)\bigr),
\end{equation}
which is referred to as the gradient dominance condition or the Polyak-Łojasiewicz inequality \cite{polyak1963gradient}. With this inequality, the gradient flow converges exponentially fast, as $\frac{{\rm d}}{{\rm d}t} h(\theta_t) \leq -2\alpha \bigl(h(\theta_t) - h(\thetapo)\bigr)$ leading to $h(\theta_t) - h(\thetapo) = \bigO\bigl(\exp(-2\alpha t)\bigr)$. In summary, the strong convexity in the Euclidean space leads to gradient dominance condition, which then implies the exponential convergence of the gradient flow.

Now, consider a Riemannian manifold $\mathcal{M}$. For $p \in \mathcal{M}$, the Riemannian metric is $g_p: T_p\mathcal{M} \times T_p\mathcal{M} \to \mathbb{R}$. Geodesics in the manifold generalizes the concept of straight lines in Euclidean spaces. A constant-speed geodesics connecting two points $p_0,p_1 \in \mathcal{M}$ can be characterized as the minimizer of the following variational problem:
\begin{equation}
    \mathcal{D}^2(p_0,p_1) = \inf_{p_t \in \mathcal{M}, v_t \in T_p\mathcal{M}} \{\int_0^1 g_{p_t}(v_t,v_t){\rm d}t: \ \frac{{\rm d}}{{\rm d}t} p_t = v_t, p_{t=0} = p_0, p_{t=1}=p_1\},
\end{equation}
where $\mathcal{D}$ is called the geodesic distance between $p_0$ and $p_1$, and $v_t$ denotes velocity. Such a geodesics satisfies $g_{p_t}(v_t,v_t) = \mathcal{D}^2(p_0,p_1)$ so is referred to as the constant-speed geodesics. 

A function $h: \mathcal{M} \to \mathbb{R}$ is called geodesically convex, if for any $p_0, p_1 \in \mathcal{M}$, $\lambda \in [0,1]$, it holds that
\begin{equation}
    h(p_\lambda) \leq \lambda h(p_0) + (1-\lambda)h(p_1),
\end{equation}
where $p_t, t\in [0,1]$ is the constant-speed geodesics connecting $p_0$ and $p_1$. 
The first order characterization of geodesic convexity is
\begin{equation}
    h(p_1) \geq h(p_0) + g_{p_0}\bigl(\text{grad}~h(p_0), v_0\bigr) + \frac{\alpha}{2}\mathcal{D}^2(p_0,p_1),
\end{equation}
where $\text{grad}$ is the gradient operator on the manifold, $v_0$ is the initial velocity of the constant-speed geodesics connecting $p_0$ and $p_1$. Here $h$ is called $\alpha$-strongly geodesically convex if $\alpha > 0$.
Furthermore, the second order characterization of the geodesic convexity is $\text{Hess}~ h(p_0) \succeq \alpha I$ where $\text{Hess}$ is the Hessian operator defined on the manifold, which satisfies
\begin{align}
         g_{p_0}\bigl(\textrm{Hess}\, h(p_0) v_0,v_0\bigr) = \frac{\dd^2 }{\dd t^2}\Big |_{\rm geod} h(p_0) 
    \end{align}
    where $p_t$ is a constant-speed geodesics that starts at $p_0$ with initial velocity vector $v_0$. Here we used the notation $\frac{\dd^2 }{\dd t^2}\Big |_{\rm geod}$ to specifically mean that the derivative is taken with respect to a constant geodesics $p_t$ at $t=0$. 

Similarly, as the Euclidean case, $\alpha$-strongly geodesic convexity implies the gradient dominance condition:
\begin{equation}
    g_p\bigl(\text{grad}\, h(p), \text{grad}\, h(p)\bigr) \geq 2\alpha \bigl(h(p) - h(p^*)\bigr),
\end{equation}
where $p^*$ is the global minimizer of $h$. Again, this inequality leads to the exponential convergence of gradient flows of $h$ on the manifold $\mathcal{M}$.

\subsection{Hessian operator in the Fisher-Rao geometry}
In this subsection, we derive the Hessian operator of energy functionals in the Fisher-Rao geometry, which will be used in the next subsection to analyze the geodesic convexity of $f$-divergences.

The Fisher-Rao metric in the probability density space $\mathcal{P}$ is
\begin{equation}
    g_{\rho}^{\mathrm{FR}}(\sigma_1,\sigma_2)=\int \frac{\sigma_1\sigma_2}{\rho} \dd\theta, \text{ for } \sigma_1,\sigma_2 \in T_{\rho}\PP.
\end{equation}
\begin{newremark}
    The Fisher-Rao metric can also be viewed as an infinite dimensional Euclidean metric restricted to the unit ball $\{\psi:\int \psi^2 = 1\}$, if we apply the transformation $\psi = \sqrt{\rho}$; see, e.g., \cite{halder2018gradient}. We note that one could derive the formulas of the geodesic convexity conditions and the relevant functional inequalities using the transformed coordinate $\psi$ by exploiting the Euclidean structure. In this paper, we choose to work with the original coordinates, as they eventually lead to equivalent formulas. Determining for which functions $f$ these conditions hold is a separate question independent from the coordinates chosen and is the one we primarily address in this work.
\end{newremark}

The geodesic distance between $\rho_0, \rho_1$, referred to as the Fisher-Rao distance, satisfies
\begin{subequations}
\label{eqn-fisher-rao-geodesic}
    \begin{align}
       \calD_{\rm FR}^2(\rho_0, \rho_1) &= \inf_{\rho_t,\sigma_t}\Bigl\{
          \int_0^1 \int \frac{|\sigma_t|^2}{\rho_t} \dd\theta \dd t:
          \partial_t \rho_t = \sigma_t
          \Bigr\}
          \\
          &= \inf_{\rho_t,\psi_t}\Bigl\{
          \int_0^1 \int \rho_t \bigl(\psi_t - \E_{\rho_t}[\psi_t]\bigr)^2 \dd\theta \dd t:
          \partial_t \rho_t = \rho_t (\psi_t - \E_{\rho_t}[\psi_t])
          \Bigr\}.
    \end{align}
\end{subequations}
\begin{newremark}
\label{remark:Fisher-Rao-vs-Hellinger}
    The Fisher-Rao geodesic distance attains an explicit formula, equivalent to (up to a constant scaling) the spherical Hellinger distance \cite{halder2018gradient, laschos2019geometric, lu2022birth}:
$\calD_{\rm FR}^2(\rho_0, \rho_1) = 4 \operatorname{arccos}^2\left(\int \sqrt{\rho_0}\sqrt{\rho_1}{\rm d}\theta\right)$.
In view of this relation, Fisher-Rao gradient flows are sometimes referred to as spherical Hellinger gradient flows in the literature \cite{lindsey2022ensemble,lu2022birth}.
    If we do not restrict the distributions $\rho_t$ to be on the probability space and we allow them to have any positive mass, then by using the relation $\frac{|\sigma_t|^2}{\rho_t} = \frac{|\partial_t{\rho}_t|^2}{\rho_t} = 4\Bigl|\frac{\mathrm{d}}{\mathrm{d}t}\sqrt{\rho_t}\Bigr|^2$
and the Cauchy-Schwarz inequality, we will get that the optimal objective value in \eqref{eqn-fisher-rao-geodesic} is $4\int |\sqrt{\rho_0} - \sqrt{\rho_1}|^2 \mathrm{d}\theta.$
This is (up to a constant scaling) the Hellinger distance \cite{gibbs2002choosing}. For a more detailed discussion of the origin and historical development of the terms Fisher–Rao and Hellinger, we refer to \cite[Remark 2.2]{mielke2025hellinger}.
\end{newremark}

Using a formal Lagrangian duality calculation, we can derive the optimality condition of the variational problem \eqref{eqn-fisher-rao-geodesic} as a coupled system of differential equations (see also \cite{wang2020information}):
\begin{equation}
\label{eq:FR-geodesics-PDE}
     \partial_t \rho_t = \rho_t \bigl(\psi_t - \E_{\rho_t}[\psi_t]\bigr),  \qquad
\partial_t \psi_t = -\frac{1}{2}(\psi_t - \E_{\rho_t}[\psi_t])^2,
\end{equation}
along with boundary condition $\rho_{t=0} = \rho_0, \rho_{t=1}=\rho_1$. This can be utilized to calculate the gradient and Hessian operators.

Consider any constant-speed geodesics starting at $\rho$ with initial velocity field $\sigma = \rho (\psi - \E_{\rho}[\psi])$. Differentiating along this geodesics leads to the formula of the gradient and Hessian operator in the Fisher-Rao geometry:
\begin{subequations}
\label{eq:grad-Hessian-EE}
    \begin{align}
        g_{\rho}^{\rm FR}(\textrm{grad}_{\rm FR} \EE(\rho), \sigma) =& \frac{\dd }{\dd t}\Big |_{\rm geod} \EE(\rho)
        = \int \frac{\delta \EE}{\delta \rho} \rho \bigl(\psi - \E_{\rho}[\psi]\bigr)\dd\theta,
        \\
         g_{\rho}^{\rm FR}(\textrm{Hess}_{\rm FR} \EE(\rho) \sigma, \sigma) =& \frac{\dd^2 }{\dd t^2}\Big |_{\rm geod} \EE(\rho) =
          \frac{\dd }{\dd t}\Big |_{\rm geod}  \int \frac{\delta \EE}{\delta \rho} \rho \bigl(\psi - \E_{\rho}[\psi]\bigr)\dd\theta \nonumber\\
            =& \int\frac{\delta^2\EE}{\delta \rho^2} \rho^2 \bigl(\psi - \E_{\rho}[\psi]\bigr)^2
            +
         \frac{1}{2}
         \rho \bigl(\psi - \E_{\rho}[\psi]\bigr)^2 \Bigl(\frac{\delta\EE}{\delta \rho} - \E_{\rho}[\frac{\delta\EE}{\delta \rho}]
         \Bigr)\dd\theta.
    \end{align}
\end{subequations}
Here, we repeatedly used the derivatives of $\rho$ and $\psi$ along constant-speed geodesics in \cref{eq:FR-geodesics-PDE}.

\subsection{Geodesic convexity in the Fisher-Rao geometry}
For the $f$-divergen\-ces~\eqref{eqn-f-divergences}, the corresponding gradient and Hessian operators are
\begin{subequations}
\label{eq:grad-Hessian-Df}
    \begin{align}
        g_{\rho}^{\rm FR}\bigl(\textrm{grad}_{\rm FR} D_f[\rho\Vert \rhopo], \sigma\bigr)
        =& \int f'\bigl(\frac{\rho}{\rhopo}\bigr) \rho \bigl(\psi - \E_{\rho}[\psi]\bigr)\dd\theta,  \label{eq:grad-Df}
        \\
         g_{\rho}^{\rm FR}\bigl(\textrm{Hess}_{\rm FR} D_f[\rho\Vert \rhopo] \sigma, \sigma\bigr)
          =& \int f''\bigl(\frac{\rho}{\rhopo}\bigr)\frac{\rho^2 \bigl(\psi - \E_{\rho}[\psi]\bigr)^2}{\rhopo}
          \label{eq:Hess-Df}  \\+&
         \frac{1}{2}
         \rho \bigl(\psi - \E_{\rho}[\psi]\bigr)^2 \Bigl(f'\bigl(\frac{\rho}{\rhopo}\bigr) - \E_{\rho}[f'\bigl(\frac{\rho}{\rhopo}\bigr)]
         \Bigr)\dd\theta,  \nonumber
    \end{align}
\end{subequations}
where $\sigma = \rho \bigl(\psi - \E_{\rho}[\psi]\bigr)$. 

We employ the aforementioned formula to investigate the geodesic convexity of $f$-divergences within the Fisher-Rao geometry. In what follows, we first construct counterexamples to demonstrate that the KL divergence fails to exhibit geodesic convexity (see \cref{theorem:FR-displacement-convex}). Subsequently, we provide a sufficient condition for an $f$-divergence to satisfy geodesic convexity (see \cref{theorem:postive-uniform-convex}).

\begin{theorem}
\label{theorem:FR-displacement-convex}
   For any smooth, positive target density $\rho^*$, the KL divergence $\mathrm{KL}[\rho\Vert \rho^*]$ is not geodesically convex within the Fisher-Rao geometry.

Furthermore, for any $\epsilon > 0$, the KL divergence is not geodesically convex in the neighborhood of $\rho^*$ defined by $\mathrm{KL}[\rho\Vert \rho^*] \leq \epsilon$.
\end{theorem}

\begin{newremark}
We note that small neighborhood in \cref{theorem:FR-displacement-convex} can also be defined using the Fisher-Rao distance $\calD_{\rm FR}^2(\rho, \rho^{*})$. This is because we have the inequality 
$\calD_{\rm FR}^2(\rho, \rho^{*}) \leq 4 \mathrm{KL}[\rho\Vert \rho^*]$. Indeed, 
\begin{align*}
    \mathcal{D}^2_{\rm FR}(\rho,\rho^{*}) &= 4\arccos^2\Bigl(\int \sqrt{\rho\rho^{*}} d\theta\Bigr) \leq 
    4\Bigl[-2\log \int \sqrt{\rho\rho^{*}}d\theta \Bigr]
    \leq 4\rm{KL}[\rho \Vert \rho^{*}],
\end{align*}
here the first inequality follows from $\arccos^2(x) \leq -2\log x \,\forall \, 0 \leq x \leq 1$ with $x = \int \sqrt{\rho\rho^{*}}{\rm d}\theta$. This can be shown by defining $g(t) = -2\log \cos t -t^2$ for $t = \arccos(x) \in [0,\frac{\pi}{2}]$, and noting that $g'(t) =2\tan t - 2t \geq 0$, which implies $g(t)\geq g(0) = 0$. The second inequality follows from the monotonicity of the Rényi divergence
$$D_\alpha[\rho\Vert\rho^*] =  \frac{1}{\alpha - 1}\log \int \rho^{\alpha} {\rho^*}^{1-\alpha}{\rm d}\theta,$$ which is nondecreasing in $\alpha$ \cite{van2010renyi}. Hence $-2\log \int \sqrt{\rho\rho^{*}}{\rm d}\theta = D_{\frac{1}{2}}[\rho\Vert\rho^*] \leq D_{1}[\rho\Vert\rho^*] = \rm{KL}[\rho \Vert \rho^{*}].$

\end{newremark}

\begin{proof}
\label{proof-formal:FR-displacement-convex}
    The KL-divergence corresponds to the $f$-divergence with $f(x) = x\log x$. In such case, the Hessian formula~\eqref{eq:Hess-Df} takes the form
    \begin{equation}
    \label{eq:Hess-KL}
        \begin{split}
            g_{\rho}^{\rm FR}\bigl(\textrm{Hess}_{\rm FR} {\rm KL}[\rho\Vert \rhopo] \sigma, \sigma\bigr)
         &= \int \rho \bigl(\psi - \E_{\rho}[\psi]\bigr)^2
         \\
         &+
         \frac{1}{2}
         \rho \bigl(\psi - \E_{\rho}[\psi]\bigr)^2 \Bigl(\log\frac{\rho}{\rhopo} - \E_{\rho}[\log\frac{\rho}{\rhopo}]
         \Bigr)\dd\theta.
        \end{split}
    \end{equation}
The positive semi-definiteness of the Hessian depends on the sign of $1+\frac{1}{2}\bigl(\log\frac{\rho}{\rhopo} - \E_{\rho}[\log\frac{\rho}{\rhopo}]\bigr)$. For instance, the KL divergence is geodesically convex when this term is non-negative. A simple sufficient condition is  $1/e\leq {\rho}/{\rhopo}\leq e$. 

However, in general, this term does not have a definite sign, and the KL divergence will not be geodesically convex.
To illustrate this, 
let us consider a specific case where $\psi = \log\frac{\rho}{\rhopo}$. In this scenario, \eqref{eq:Hess-KL} becomes
\begin{equation}
\label{eq:Hess-KL-GF}
     \calH(\rho, \rhopo)
         = \int \rho \Bigl(\log\frac{\rho}{\rhopo} - \E_{\rho}[\log\frac{\rho}{\rhopo}]
         \Bigr)^2
         +
         \frac{1}{2}
         \rho \Bigl(\log\frac{\rho}{\rhopo} - \E_{\rho}[\log\frac{\rho}{\rhopo}]
         \Bigr)^3\dd\theta.
\end{equation}
We will now present two examples to demonstrate why the above term is not generally non-negative.

We begin by considering a simple pedagogical counterexample where 
$\rhopo(\theta) = \N(\theta; 0, 1)$ represents a 1D Gaussian density. We then construct $\rho(\theta) = \N(\theta; \mu, \sigma^2)$ as another Gaussian distribution. The following calculations can be made:
\begin{equation}
\label{eq:Gaussian-KL}
  \begin{split}
 &\E_{\rho}[\log\frac{\rho}{\rhopo}]= \frac{\sigma^2}{2} + \frac{\mu^2}{2}-\log \sigma - \frac{1}{2},
 \\
    &\int\rho\Bigl(\log\frac{\rho}{\rhopo} -  \E_{\rho}[\log\frac{\rho}{\rhopo}]\Bigr)^2 \dd\theta  = \frac{1}{2}(\sigma^2 - 1)^2 + \sigma^2\mu^2,
    \\
    &\int\rho\Bigl(\log\frac{\rho}{\rhopo} -  \E_{\rho}[\log\frac{\rho}{\rhopo}]\Bigr)^3 \dd\theta  = (\sigma^2 - 1)^3 + 3(\sigma^2-1)\sigma^2\mu^2.
\end{split}
\end{equation}
Plugging \eqref{eq:Gaussian-KL} into \eqref{eq:Hess-KL-GF}, we get
\begin{equation*}
     \calH(\rho, \rhopo)
         = \frac{1}{2}\sigma^2(\sigma^2 - 1)^2
         + \frac{3}{2}\sigma^4\mu^2
         - \frac{1}{2}\sigma^2\mu^2 =
         \frac{\sigma^2}{2}\bigl((\sigma^2 - 1)^2 + 3\sigma^2\mu^2 - \mu^2\bigr).
\end{equation*}
We can choose $\sigma^2 < \frac{1}{3}$ and $\mu^2 > \frac{(\sigma^2 - 1)^2}{1 - 3\sigma^2}$, which leads to $\calH(\rho, \rhopo) < 0$.

Next, we present a general example beyond the Gaussian case to demonstrate the theorem. We will first construct a counterexample where the density is not continuous. We then mollify this example to obtain a smooth density.

Given any $\rhopo$ and radius $r>0$, we denote
\begin{align}\label{e:defrhor}
\rho_r := \int_{B_r} \rhopo(\theta) \dd\theta,
\end{align}
and construct the following density
\begin{equation}
\label{eq:rho-discontinuous}
    \rho(\theta) = \rhopo(\theta)(x_1 \mathbbm{1}_{B^c_r}(\theta) + x_2 \mathbbm{1}_{B_r}(\theta))
\end{equation}
where $\mathbbm{1}_A$ is the indicator function of a set $A$ and $B_r$ denotes the ball with radius $r$. For $\rho(\theta)$ to be a probability density, it is necessary that
\begin{align}\label{e:x1x2condition}
    x_1 \Bigl(1 - \int_{B_r} \rhopo(\theta) \dd\theta\Bigr) + x_2 \int_{B_r} \rhopo(\theta) \dd\theta = x_1(1-\rho_r)+x_2\rho_r=1.
\end{align}
In fact, for any choice of $x_1 > 1 > x_2 > 0$, we can take some $r\in (0,+\infty)$ such that $\rho_r  = \frac{x_1 - 1}{x_1 - x_2} \in (0,1)$
which leads to \eqref{e:x1x2condition}.

We can calculate KL$[\rho\Vert \rhopo]$ and $\calH(\rho, \rhopo)$ denoted in \eqref{eq:Hess-KL-GF}, as follows:
\begin{align*}
       {\rm KL}[\rho\Vert \rhopo] = & \frac{x_1 - x_1x_2}{x_1 - x_2}\log x_1 + \frac{x_1x_2 - x_2}{x_1 - x_2}\log x_2, \\
       \calH(\rho, \rhopo) = & \rho_r(1-\rho_r)x_1x_2\Bigl(\log\frac{x_1}{x_2}\Bigr)^2 \!\!+ \frac{1}{2}\rho_r(1-\rho_r)x_1x_2\bigl(x_2\rho_r - x_1(1-\rho_r)\bigr)\Bigl(\log\frac{x_1}{x_2}\Bigr)^3  \nonumber\\
       = & \rho_r(1-\rho_r)x_1x_2\Bigl(\log\frac{x_1}{x_2}\Bigr)^2\Bigl(1 + \frac{x_2\rho_r - x_1(1-\rho_r)}{2}\log\frac{x_1}{x_2}\Bigr) \nonumber\\
       = & \rho_r(1-\rho_r)x_1x_2\Bigl(\log\frac{x_1}{x_2}\Bigr)^2\Bigl(1 - \frac{x_1 + x_2 - 2x_1x_2}{2(x_1 - x_2)}\log\frac{x_1}{x_2}\Bigr).
\end{align*}
Let us choose $x_1 = e^{\epsilon}$ with $\epsilon>0$ and $x_2 = e^{-M}$ with $M > 4$. Then we have
\begin{equation}
 \label{eq:H-discontinuous}
    \begin{aligned}
   {\rm KL}[\rho\Vert \rhopo] \leq & \frac{x_1 - x_1x_2}{x_1 - x_2}\log x_1 \leq \log x_1 = \epsilon, \\
       \calH(\rho, \rhopo)
       = & \rho_r(1-\rho_r)x_1x_2\Bigl(\log\frac{x_1}{x_2}\Bigr)^2\Bigl(1 - \frac{x_1 + x_2 - 2x_1x_2}{2(x_1 - x_2)}(M + \epsilon)\Bigr)
       \\
       = & \rho_r(1-\rho_r)x_1x_2\Bigl(\log\frac{x_1}{x_2}\Bigr)^2\Bigl(1 - \frac{e^{M} + e^{-\epsilon} - 2}{2(e^{M} - e^{-\epsilon})}(M + \epsilon)\Bigr) 
       \\
       \leq &\rho_r(1-\rho_r)x_1x_2\Bigl(\log\frac{x_1}{x_2}\Bigr)^2\Bigl(1 - \frac{e^M - 2}{2e^M}M \Bigr)<0. 
    \end{aligned}
\end{equation}
Therefore, the Hessian is not positive definite.
Finally, we can smooth $\rho$ in \eqref{eq:rho-discontinuous} by introducing the mollifier:
\begin{equation*}
    \phi_\delta(\theta) = \begin{cases}
       \frac{1}{Z_\delta}e^{-\frac{1}{1 - |\theta/\delta|^2}}  & |\theta|< \delta \\
       0  & |\theta| \geq \delta,
   \end{cases}
\end{equation*}
which has a compact support within $B_\delta$ and satisfies $\int \phi_{\delta}(\theta) \dd\theta = 1$.
The smoothed density is constructed using convolution, through
\begin{equation*}
    \rho_\delta(\theta) = \rhopo(\theta)  \bigl(\phi_\delta  * (x_1 \mathbbm{1}_{B^c_{r+\delta}} + x_2 \mathbbm{1}_{B_{r-\delta}}  + x_3 \mathbbm{1}_{B_{r+\delta} \backslash B_{r-\delta}} )\bigr)(\theta).
\end{equation*}
Here, $x_3$ is a scalar determined by $\int \rho_\delta(\theta)\dd\theta = 1$. Note that $x_1 > 1 > x_2 > 0$. It holds that $x_2 \leq  x_3 \leq x_1$
and $x_2 \leq \frac{\rho_\delta}{\rhopo} \leq  x_1$.
Since $\rho_\delta(\theta)$ only deviates from  $\rho$ in \eqref{eq:rho-discontinuous} within the compact band $B_{r+2\delta} \backslash B_{r-2\delta}$, the last inequality in~\cref{eq:H-discontinuous} remains valid when $\delta$ is sufficiently small.
\end{proof}

\cref{theorem:FR-displacement-convex} demonstrates that the KL divergence, a specific instance of an $f$-divergence, is not geodesically convex. This finding implies that $f$-divergences generally do not possess this property. Consequently, we cannot rely on convexity to establish exponential convergence for Fisher-Rao gradient flows of the KL divergence. In \cref{sec-Functional Inequality: Dual Gradient Dominance Condition}, we will introduce alternative conditions that enable us to prove convergence results.

Conversely, we will now establish a sufficient condition on $f$ that ensures the geodesic convexity of its corresponding $f$-divergence.
\begin{theorem}
\label{theorem:postive-uniform-convex}
    Assume  $f\in C^2(0,+\infty)$ is convex, satisfies $f(1)=0$  and $xf'(x)$ is concave.
    Then, for any smooth, positive target density $\rho^*$, the $f$-divergence is geodesically convex within the Fisher-Rao geometry. Furthermore, when $f''(1) > 0$, such $f$-divergence is $\alpha_f$-strongly geodesically convex for any $\rhopo$, where $\alpha_f > 0$ is a constant that depends solely on $f$. These assumptions are satisfied for $f''(x) = x^p$ with $p\leq -2$.
\end{theorem}
\begin{proof}

When $xf'(x)$ is concave, using Jensen's inequality we get
\begin{align}
\E_{\rho}[f'\bigl(\frac{\rho}{\rhopo}\bigr)] = \E_{\rhopo}\Bigl[\frac{\rho}{\rhopo}f'\bigl(\frac{\rho}{\rhopo}\bigr)\Bigr]
         \leq \E_{\rhopo}\bigl[\frac{\rho}{\rhopo}\bigr]f'\Bigl(\E_{\rhopo}\bigl[\frac{\rho}{\rhopo}\bigr]\Bigr)= f'(1). \label{eq:xf'(x)}
\end{align}
We recall the expression of Hessian operator from \eqref{eq:Hess-Df}
\begin{align}\begin{split}\label{e:gHess}
&\phantom{{}={}}g_{\rho}^{\rm FR}\bigl(\textrm{Hess}_{\rm FR} D_f[\rho\Vert \rhopo] \sigma, \sigma\bigr) \\
          &= \frac{1}{2}\int
         \rho \bigl(\psi - \E_{\rho}[\psi]\bigr)^2 \Bigl(2\frac{\rho}{\rhopo}f''\bigl(\frac{\rho}{\rhopo}\bigr)
         +f'\bigl(\frac{\rho}{\rhopo}\bigr) - \E_{\rho}f'\bigl(\frac{\rho}{\rhopo}\bigr)
         \Bigr)\dd\theta
         \\
         &\geq \frac{1}{2}\int
         \rho \bigl(\psi - \E_{\rho}[\psi]\bigr)^2 \Bigl(2\frac{\rho}{\rhopo}f''\bigl(\frac{\rho}{\rhopo}\bigr)
         +f'\bigl(\frac{\rho}{\rhopo}\bigr) - f'(1)
         \Bigr)\dd\theta,
\end{split}\end{align}
where in the last line we used \eqref{eq:xf'(x)}. We introduce the function
\begin{align}\label{e:defh}
    h(x) = 2xf''(x) + f'(x) - f'(1)=2\bigl(xf'(x)\bigr)'-f'(x)-f'(1).
\end{align}
With the function $h(x)$, we can rewrite \eqref{e:gHess} as
\begin{align}\label{eq:h-hess-Df}
    g_{\rho}^{\rm FR}\bigl(\textrm{Hess}_{\rm FR} D_f[\rho\Vert \rhopo] \sigma, \sigma\bigr) \geq \frac{1}{2}\int
         \rho (\psi - \E_{\rho}[\psi])^2 h\bigl(\frac{\rho}{\rhopo}\bigr) \dd\theta
\end{align}
Due to our assumptions, $xf'(x)$ and $-f(x)$ are both concave, so $(xf'(x))'$ and $-f'(x)$ are non-increasing functions of $x$. It follows that $h(x)$ in \eqref{e:defh} is non-increasing.
Therefore, when $x \leq 1$, we have
$h(x) \geq h(1) = 2f''(1)\geq 0,$
and when $x \geq 1$, we have
$h(x) \geq  f'(x) - f'(1) \geq 0.$
Thus, $h(x) \geq 0$. This together with \eqref{eq:h-hess-Df} implies that
$$g_{\rho}^{\rm FR}\bigl(\textrm{Hess}_{\rm FR} D_f[\rho\Vert \rhopo] \sigma, \sigma\bigr)\geq 0,$$
for any tangent vector $\sigma$. This shows that the $f$-divergence is geodesically convex.

Next we prove that when $f''(1) > 0$, the $f$-divergence is $\alpha_f$-strongly geodesically convex. To do so, we need a stronger estimate of the function $h(x)$.
When $f''(1) > 0$, there exists
$\delta_f > 0$, such that for $x \in [1, 1+\delta_f]$,
    $\frac{f''(1)}{2} \leq f''(x)$.
This inequality, combined with the preceding discussion, implies that
\begin{equation}
\label{eq:h}
    h(x) \geq
    \begin{cases}
    2f''(1) & 0<x\leq 1\\
    2xf''(x)  \geq  f''(1)&  1 < x \leq 1 + \delta_f \\
    f'(x) - f'(1) = \int_1^x f''(\xi){\rm d}\xi \geq \delta_f \frac{f''(1)}{2} & x\geq 1 + \delta_f
    \end{cases}.
\end{equation}
Consequently, $h(x) \geq \alpha_f > 0$, where $\alpha_f = \min\{f''(1), \delta_f\frac{f''(1)}{2}\}$ is a constant that depends only on $f$.
Thus, when $f''(1) > 0$, combining \eqref{eq:h} and \eqref{eq:h-hess-Df}, we conclude that the $f$-divergence is $\alpha_f$-strongly geodesically convex, i.e.,
\begin{equation*}
   g_{\rho}^{\rm FR}(\textrm{Hess}_{\rm FR} D_f[\rho\Vert \rhopo] \sigma, \sigma) \geq \frac{\alpha_f}{2}g_{\rho}^{\rm FR}(\sigma, \sigma).
\end{equation*}
Finally, for $f''(x) = x^p$ with $p \leq -2$, $f$ is convex on $(0,+\infty)$. Moreover, $xf'(x)$ is concave, since $(xf'(x))'' = 2f''(x) + xf'''(x) = (2+p)x^p \leq 0$ for $x \in (0,+\infty)$. Thus, all assumptions are satisfied for $f$ with $f''(x) = x^p$ where $p \leq -2$.
\end{proof}

\section{Functional Inequality: Gradient Dominance Condition}
\label{sec-Functional Inequality: Gradient Dominance Condition}
In this section, we study the gradient dominance condition \eqref{eq:intro-GDC-functional} of the $f$-divergences in the Fisher-Rao geometry. Counterexamples will be presented to illustrate that the KL-divergence does not satisfy the gradient dominance condition. Subsequently, we provide necessary and sufficient conditions to identify the functions $f$ that ensure the gradient dominance condition is met.

First, we calculate the norm of the gradient in the Fisher-Rao geometry
\begin{align*}
&g_{\rho}^{\mathrm{FR}}(\textrm{grad}_{\rm FR}D_f[\rho\Vert \rhopo], \textrm{grad}_{\rm FR}D_f[\rho\Vert \rhopo])\\
    =&g_{\rho}^{\mathrm{FR}}\left( -\rho f'\bigl(\frac{\rho}{\rhopo}\bigr) + \rho \E_{\rho}\Bigl[ f'\bigl(\frac{\rho}{\rhopo}\bigr)\Bigr] , -\rho f'\bigl(\frac{\rho}{\rhopo}\bigr) + \rho \E_{\rho}\Bigl[ f'\bigl(\frac{\rho}{\rhopo}\bigr)\Bigr]\right)\\
    =&\int \rho\left(  f'\bigl(\frac{\rho}{\rhopo}\bigr) -  \E_{\rho}\Bigl[ f'\bigl(\frac{\rho}{\rhopo}\bigr)\Bigr]\right)^2 \dd \theta.
\end{align*}
The gradient dominance condition can be formulated as follows:
\begin{condition}[Gradient dominance condition]\label{cond:gradient-dominant-condition-f}
The function $f$ is said to satisfy the gradient dominance condition if there exists a constant $\alpha_f > 0 $, depending only on $f$, such that
    \begin{equation}
\label{eq:gradient-dominant-condition-f}
\begin{aligned}
\int \rho\left(  f'\bigl(\frac{\rho}{\rhopo}\bigr) -  \E_{\rho}\Bigl[ f'\bigl(\frac{\rho}{\rhopo}\bigr)\Bigr]\right)^2 \dd \theta
 \geq  \alpha_f D_f[\rho \Vert \rhopo],
\end{aligned}
\end{equation}
for any smooth positive densities $\rho$ and $\rhopo$, provided that $f'(\frac{\rho}{\rhopo})$ is absolutely integrable with respect to $\rho$.
\end{condition}

\subsection{KL divergence: no gradient dominance}
We first present the following negative result concerning the gradient dominance condition for the KL divergence. In this case, $f(x) = x\log x$ and $f'(x) = \log x + 1$.
\begin{theorem}
\label{theorem:FR-gradient-dominant-condition}
    Given any smooth, positive target distribution $\rhopo$ and any $\epsilon > 0$, we can construct $\rho$, such that $ \rho\neq \rhopo$ and
\begin{align}
\label{eq:log-Sol-KL}
    \int\rho\Bigl(\log\frac{\rho}{\rhopo} -  {\rm KL}[\rho\Vert\rhopo]\Bigr)^2 \dd\theta \leq \epsilon \rm{KL}[\rho \Vert \rhopo].
\end{align}
Moreover,  given any $\epsilon' > 0$, we can also construct $\rho$ in the neighborhood of $\rhopo$ defined as {\rm KL}$[\rho \Vert \rhopo] < \epsilon'$, such that \eqref{eq:log-Sol-KL} holds.
\end{theorem}
\begin{proof}
We denote
    $\G(\rho, \rhopo) = \int\rho\Bigl(\log\frac{\rho}{\rhopo} -  {\rm KL}[\rho\Vert\rhopo]\Bigr)^2 \dd\theta.$
We first consider simple 1D Gaussian densities as a pedagogical example:  $\rho(\theta) = \N(\theta; \mu, \sigma)$ and $\rhopo(\theta) = \N(\theta; 0, 1)$. We can calculate both sides of \eqref{eq:log-Sol-KL} similar to \cref{eq:Gaussian-KL}:
\begin{align*}
    \G(\rho, \rhopo) = \frac{1}{2}\Bigl(\sigma^2 - 1\Bigr)^2 + \sigma^2\mu^2, \qquad {\rm KL}[\rho\Vert\rhopo] = \frac{\sigma^2}{2} + \frac{\mu^2}{2}-\log \sigma - \frac{1}{2}.
\end{align*}
We choose $M > \max\{1, \sqrt{\frac{3}{\epsilon}}\}$ and set $\sigma = \frac{1}{M}$ and $ \mu = M$, then
\begin{align*}
    \frac{\G(\rho, \rhopo)}{{\rm KL}[\rho\Vert\rhopo]}  = \frac{\frac{1}{2}\Bigl(\frac{1}{M^2} - 1\Bigr)^2 + 1}{\frac{1}{2M^2} + \log M - \frac{1}{2} + \frac{M^2}{2}} \leq \frac{3}{M^2} \leq \epsilon,
\end{align*}
and thus, the gradient dominance condition is violated. Here we used the fact that $\frac{1}{2M^2} + \log M \geq \frac{1}{2}$ when $M \geq 1$.

Then, we consider the general case. Given any $\rhopo$, we construct
\begin{equation}
\label{eq:rho-discontinuous2}
    \rho(\theta) = \rhopo(\theta)\bigl(x_1 \mathbbm{1}_{B^c_r}(\theta) + x_2 \mathbbm{1}_{B_r}(\theta)\bigr).
\end{equation}
Then, similar to the proof of \Cref{theorem:FR-displacement-convex}, we have the flexibility to select any $x_1 > 1 > x_2 > 0$; in particular, we can adjust $r$ to ensure $\rho$ is a density function.

We calculate both sides of \cref{eq:log-Sol-KL} as follows:
    \begin{align*}
      &\G(\rho, \rhopo)  =  x_1x_2\frac{(x_1-1)(1-x_2)}{(x_1 - x_2)^2}\Bigl(\log\frac{x_1}{x_2}\Bigr)^2  \\
      &{\rm KL}[\rho\Vert \rhopo] =  \frac{x_1 - x_1x_2}{x_1 - x_2}\log x_1 + \frac{x_1x_2 - x_2}{x_1 - x_2}\log x_2.
    \end{align*}
If we choose $x_1 = e^{\epsilon'}$ with $\epsilon'>0$ and $x_2 = e^{-M}$ with $M > \max\{\epsilon,\,\epsilon',\,2\log\frac{1}{\epsilon\epsilon'},\,20\}$, we have
\begin{subequations}
\label{eq:G/KL-discontinuous}
    \begin{align}
   {\rm KL}[\rho\Vert \rhopo] \leq & \frac{x_1 - x_1x_2}{x_1 - x_2}\log x_1 \leq \log x_1 = \epsilon'\\
       \frac{\G(\rho, \rhopo)}{{\rm KL}[\rho\Vert \rhopo]}
       = & \frac{(1-\frac{1}{x_1})(\frac{1}{x_2}-1)\bigl(\log \frac{x_1}{x_2}\bigr)^2}{(\frac{1}{x_2}-\frac{1}{x_1})\bigl((\frac{1}{x_2}-1)\log x_1 + (1-\frac{1}{x_1})\log x_2\bigr)}
       \\
       = & \frac{(1-e^{-\epsilon'})(e^M-1)(M+\epsilon')^2}{(e^M-e^{-\epsilon'})\bigl((e^M-1)\epsilon' - (1-e^{-\epsilon'})M\bigr)}
       \nonumber
       \\
       < &
       \frac{(M+\epsilon')^2}{(e^M-1)\epsilon' - (1-e^{-\epsilon'})M} <
       \frac{4M^2}{\epsilon' e^M- 2M} \epsilon.
       \nonumber
\end{align}
\end{subequations}
In the last inequality, we used that when $M \geq 20$, $e^{M/2} \geq 6M^2$, and 
$4M^2 \leq \epsilon\epsilon'e^{M/2} 6M^2 - 2M^2 \leq (\epsilon e^M - 2M)\epsilon'$.

Lastly, by applying the same mollifier approach used in the proof of \Cref{theorem:FR-displacement-convex}, we can smooth $\rho$ in \eqref{eq:rho-discontinuous2} to obtain an analogous result. This completes the proof.
\end{proof}

\begin{newremark}
We also note that \cite[Theorem 3.3]{lu2019accelerating}\cite[Theorem 2.4]{lu2022birth} provide a sufficient condition for a density $\rho$ to satisfy the gradient dominance condition, namely that $\inf_\theta \frac{\rho(\theta)}{\rhopo(\theta)} > 0$, i.e., that the ratio has a positive lower bound. 
Our counterexamples here do not satisfy this condition. Specifically, for the sequence of densities $\rho$ parameterized by $M$, the constant for the gradient dominance condition worsens and tends to infinity as $M \rightarrow \infty$. This blow-up can still occur even when $\rho$
remains within an $\epsilon'$-neighborhood of $\rho^{*}$.

\end{newremark}

\subsection{Sufficient and necessary conditions for gradient dominance}
In this section, we present a sufficient and necessary condition on $f$ so that the gradient dominance is satisfied.

Since both sides of the gradient dominance condition \eqref{eq:gradient-dominant-condition-f} remain the same under the transformation $f(x) \rightarrow f(x) + c(x - 1)$ for any $c$, we can further assume $f'(1) = 0$.

\begin{condition}
    \label{cond:n=2-GDC}
    There exists an $\alpha > 0$, such that
       \begin{equation}
\label{eq:n=2-GDC}
\begin{split}
 \Bigl(f'(y) - f'(x)\Bigr)^2  \geq  \alpha(\frac{1}{x} - \frac{1}{y})\Bigl( \frac{f(x)}{1 - x} - \frac{f(y)}{1 - y}\Bigr)
\end{split}
\end{equation}
holds for any $0 < x < 1 < y$.
    \end{condition}
    
\begin{theorem}
\label{theorem:FR-gradient-dominant-condition-sn}
    Given a convex function $f \in C^2(0,+\infty)$ with $f(1) = 0$ and $f'(1) = 0$, then \Cref{cond:n=2-GDC} is sufficient and necessary for gradient dominance (\Cref{cond:gradient-dominant-condition-f}) to hold.
\end{theorem}
\begin{proof}[Sketch of Proof]
The proof breaks into the following three steps. In the first step, we use measure theoretical tools to reduce the gradient dominance condition to an inequality on the discrete probability simplex, which is simpler to deal with. In the second step, we use novel \textit{linear programming techniques} to reduce the global nonlocal term in the inequality
to a three-point interaction term. Finally, we show the two-point version of inequality, which is the sufficient and necessary condition, implies the three-point ones, thus completing the proof. We present a sketch of the main techniques here and refer to the Appendix for detailed proofs of the arguments.
\vspace{.4em}
\paragraph{\bf Step 1: Reducing to an inequality on the discrete probability simplex} We need the following lemma, whose proof is in \cref{proof:FR-gradient-dominant-condition-sn}. 
\begin{lemma}
\label{lemma-localization-of-density}
    Given smooth densities $\rho$ and $\rho^*$ defined on $\mathbb{R}^d$, there exists a sequence of bounded domains $\Omega_n \subset \mathbb{R}^d$ for $n\geq 1$ such that 
    \begin{equation}
        \Omega_n \subset \Omega_{n+1}, \quad \bigcup_{n\geq 1} \Omega_n = \mathbb{R}^d, \quad \text{and}\quad \int_{\Omega_n} \rho(\theta) {\rm d}\theta = \int_{\Omega_n} \rho^*(\theta) {\rm d}\theta.
    \end{equation}
\end{lemma}
Based on \cref{lemma-localization-of-density}, consider the truncated and normalized densities
\begin{equation}
    \rho_n(\theta) = \frac{1}{\int_{\Omega_n} \rho(\theta') {\rm d}\theta'}\rho(\theta)\mathbbm{1}_{\Omega_n}(\theta)\quad \text{and} \quad \rho^*_n(\theta) = \frac{1}{\int_{\Omega_n} \rho^*(\theta') {\rm d}\theta'}\rho^*(\theta)\mathbbm{1}_{\Omega_n}(\theta).
\end{equation}
Claim: it suffices to show for all $n\geq 1$, the following holds
\begin{equation}
\label{eqn-gradient-dominance-bounded}
    \int \rho_n\Bigl(  f'\bigl(\frac{\rho_n}{\rho^*_n}\bigr) -  c\Bigr)^2 \dd \theta
 \geq  \alpha_f D_f[\rho_n \Vert \rho^*_n],
\end{equation}
for any constant $c \in \mathbb{R}$. The rationale behind this claim is that, once the above holds, we can take $n \to \infty$. Here we have
\begin{equation}
    \lim_{n \to \infty} \int \rho_n\Big(f'\bigl(\frac{\rho_n}{\rho^*_n}\bigr) -  c\Big)^2 \dd \theta = \frac{\lim_{n \to \infty} \int_{\Omega_n} \rho \Big(f'\bigl(\frac{\rho}{\rho^*}\bigr) -  c\Big)^2 \dd \theta}{\lim_{n \to \infty} \int_{\Omega_n} \rho(\theta) {\rm d}\theta} = \int \rho\Bigl(  f'\bigl(\frac{\rho}{\rhopo}\bigr) -  c\Bigr)^2,
\end{equation}
where the convergences of the numerator and denominator are due to the monotone convergence theorem. Moreover, by Fatou's lemma, we have 
$\underline{\lim}_{n\to \infty} D_f[\rho_n \Vert \rho^*_n] \geq D_f[\rho \Vert \rho^*]$. Combining the two limits together, we get 
\begin{equation}
    \int \rho\Bigl(  f'\bigl(\frac{\rho}{\rhopo}\bigr) -  c\Bigr)^2 \geq \alpha_f D_f[\rho \Vert \rho^*].
\end{equation}
By setting $c = \E_{\rho}[f'(\frac{\rho}{\rho^{*}})]$ to achieve the minimum on the left hand side, we arrive at the desired gradient dominance condition. 

Therefore, it suffices to prove the claim in \eqref{eqn-gradient-dominance-bounded}, which is the same inequality as \eqref{eq:gradient-dominant-condition-f} but applied to densities in bounded domains. Suppose $\Omega_n = \cup_{k=1}^K \omega_k$ is partitioned to $K$ disjoint subdomains. We can approximate $\rho_n$ and $\rho_n^*$ by 
$\sum_{k=1}^K \frac{\int_{\omega_k} \rho_n {\rm d}\theta}{\int_{\omega_k} {\rm d}\theta} \mathbbm{1}_{\omega_k}$ and $\quad \sum_{k=1}^K \frac{\int_{\omega_k} \rho_n^* {\rm d}\theta}{\int_{\omega_k} {\rm d}\theta} \mathbbm{1}_{\omega_k}$.
If we can prove the inequality for the above approximate densities, we can then take finer and finer partitions and as $K \to \infty$, we will get the corresponding inequality for $\rho_n$ and $\rho^*$. This is true because we work in a bounded domain $\Omega_n$, for which we can apply uniform convergence to establish the limit; note that $\rho$ and $\rho^*$ are assumed smooth.

As a consequence, it suffices to prove the inequality when the densities are summation of indicator functions. By an abuse of notations, we denote the corresponding mass of the above approximations in each support of the indicator functions by $\rho_i, \rho_i^*$ for $1\leq i \leq K$. Furthermore, denote by $x_i = \frac{\rho_i}{\rho_i^*}$. It suffices to show
\begin{align}
\label{eqn-discrete-inequality}
\sum_{i=1}^K\rho_i \Bigl( f'(x_i) - \sum_{j=1}^K \rho_j f'(x_j)   \Bigr)^2 \geq \alpha_f \sum_i \frac{\rho_i}{x_i} f(x_i),
\end{align}
for any $\{\rho_i\}_{i=1}^K$ and $\{x_i\}_{i=1}^K$ that satisfy
\begin{align}
\label{eq:GDC-Df-discretized-constraints}
    &\sum_{i=1}^K \rho_i = 1,  \quad \sum_{i=1}^K \frac{\rho_i}{x_i} = 1, \quad
    \rho_i > 0, \quad x_i > 0.
\end{align}
This can be viewed as a reduced gradient dominance condition for probability densities on a discrete simplex.
\vspace{.4em}
\paragraph{\bf Step 2: Reducing the nonlocal term to a three-point interaction}
First, we derive an equivalent form of the inequality. For the left hand side, we have
\begin{equation}
\begin{aligned}
    &\sum_{i=1}^K\rho_i \bigl( f'(x_i) - \sum_{j=1}^K \rho_j f'(x_j)   \bigr)^2 =\sum_{i=1}^K\rho_i \bigl(f'(x_i)\bigr)^2 -  \bigl(\sum_{j=1}^K \rho_j f'(x_j)\bigr)^2\\
    = &\frac{1}{2}\left(\sum_{i=1}^K\rho_i \bigl(f'(x_i)\bigr)^2 + \sum_{j=1}^K\rho_j \bigl(f'(x_j)\bigr)^2 - 2\bigl(\sum_{i=1}^K \rho_j f'(x_i)\bigr)\bigl(\sum_{j=1}^K \rho_j f'(x_j)\bigr) \right)\\
    = & \frac{1}{2}\sum_{1\leq i,j \leq K}\rho_i\rho_j \bigl(f'(x_i) - f'(x_j)\bigr)^2.
\end{aligned}
\end{equation}
For the right hand side, we get
\begin{align}\begin{split}
    \sum_{i=1}^K \frac{\rho_i}{x_i} f(x_i)
    &=\sum_{i=1}^K \frac{\rho_i}{x_i(1-x_i)} f(x_i)-\sum_{i=1}^K \frac{\rho_i}{1-x_i} f(x_i)\\
    &=\sum_{1\leq i,j \leq K} \rho_j \frac{ \rho_i}{x_i(1-x_i)} f(x_i)-\sum_{1\leq i,j \leq K} \frac{\rho_j}{x_j}\frac{\rho_i}{1-x_i} f(x_i)\\
    &=\sum_{1\leq i,j \leq K}  \rho_i \rho_j (\frac{ 1}{x_i} - \frac{ 1}{x_j} ) \frac{f(x_i)}{1-x_i} \\
    &= \frac{1}{2}\sum_{1\leq i,j\leq K} \rho_j \rho_i  (\frac{1}{x_j} - \frac{1}{x_i})\Bigl( \frac{f(x_j)}{1 - x_j} - \frac{f(x_i)}{1 - x_i} \Bigr),
\end{split}\end{align}
where to get the second line we used \eqref{eq:GDC-Df-discretized-constraints}. We remark that since $f\in C^2(0,+\infty)$ and $f(1)=0$, $f(x)/(1-x)$ is well defined and there is no singularity at $x=1$ (we define the value here through the limit $x\to 1$). Therefore, the equivalent inequality is
\begin{align}
\label{eq:gradient-dominant-condition-f0}
 \sum_{1\leq i,j\leq K}  \rho_i\rho_j\Bigl(f'(x_i)-f'(x_j)\Bigr)^2 \geq \alpha_f\sum_{1\leq i,j\leq K}  \rho_i\rho_j(\frac{1}{x_i} - \frac{1}{x_j}) \Bigl(\frac{f(x_i)}{1-x_i } - \frac{f(x_j)}{1-x_j}\Bigr) .
\end{align}
When $K=2$, denote $x = x_1$ and $y = x_2$ with $0 < x< 1< y$. Solving \cref{eq:GDC-Df-discretized-constraints}, we obtain $\rho_1 = \frac{x(y-1)}{y-x}$ and $\rho_2 = \frac{y(1-x)}{y-x}$. Then \cref{eq:gradient-dominant-condition-f0} becomes to 
\begin{equation}
\begin{split}
\exists\ \alpha> 0 \textrm{ s.t. }
& \bigl(f'(y) - f'(x)\bigr)^2  \geq  \alpha(\frac{1}{x} - \frac{1}{y})\Bigl( \frac{f(x)}{1 - x} - \frac{f(y)}{1 - y}\Bigr),  \ \forall \ 0 < x< 1< y
\end{split}
\end{equation}
with $\alpha = \alpha_f$. This is the same as \cref{cond:n=2-GDC}. It represents a special case for the discrete inequality \cref{eqn-discrete-inequality} with $K=2$ and thus provides a necessary condition for the discrete inequality. It is also necessary for the continuous inequality \cref{eq:gradient-dominant-condition-f} by smoothing the indicator functions. It remains to show that \cref{cond:n=2-GDC} is also sufficient, i.e., the case $K=2$ holds implies that the inequality \cref{eq:gradient-dominant-condition-f0} holds for any $K\geq 3$.

We first show that cases $K \geq 3$ can be reduced to $K=3$. To do so, we use the original formulation \eqref{eqn-discrete-inequality}. We fix $x_i > 0$, $\alpha_f > 0$, and $M = \sum_{i=1}^K \rho_i f'( x_i )$, and consider the following linear programming problem with respect to variables $\{\rho_i\}_{i=1}^K$ 
\begin{subequations}
\begin{align}
    \min  &\sum_{i=1}^K \rho_i \Bigl(f'(x_i)^2 - \alpha_f  \frac{f(x_i)}{x_i}\Bigr) - M^2  \label{eq:lp-obj}
    \\
\textrm{s.t.}& \sum_{i=1}^K \rho_i = 1,\  \sum_{i=1}^K \frac{\rho_i}{x_i} = 1, \ \sum_{i=1}^K \rho_i f'( x_i ) = M,\  \rho_i \geq 0. \label{eq:lp-constraints}
\end{align}
\end{subequations}
The linear programming problem has a feasible solution and the constraint set is bounded. We claim that there is a minimizer that has at most $3$ nonzero $\rho_i$. Otherwise, without loss of generality, we assume for the minimizer $\rho$, it holds $\rho_j > 0$ for $1 \leq j \leq 4$. 
We can perturb $\rho$ in the direction $\vec{d} = [d_1,\,d_2,\,d_3,\,d_4]$ satisfying $\sum_{j=1}^K d_j = 0$, $\sum_{j=1}^K \frac{d_j}{x_j} = 0$, and $\sum_{j=1}^K d_j f'(x_j) = 0$. Since there are three linear equations in four variables, there exists a nonzero solution $\vec{d}$. Then, for any $\epsilon$, the perturbed vector $\rho + \epsilon\vec{d}$ still satisfies the constrains \cref{eq:lp-constraints}. By treating both the objective function in \cref{eq:lp-obj} and $\rho$ as functions of $\epsilon$, and choosing $\epsilon$ appropriately (either negative or positive), we can decrease the objective function or increase the number of zero entries in $\rho$. 
This leads to a contradiction so the claim is true. This means that we only need to prove the inequality for the case $K=3$. Using the equivalent formulation \eqref{eq:gradient-dominant-condition-f0}, assuming \eqref{eq:n=2-GDC} holds, we need to show
\begin{align}
\label{eq:3-point-inequality}
 \sum_{1\leq i,j\leq 3}  \rho_i\rho_j\Bigl(f'(x_i)-f'(x_j)\Bigr)^2 \geq \alpha_f\sum_{1\leq i,j\leq 3}  \rho_i\rho_j(\frac{1}{x_i} - \frac{1}{x_j}) \Bigl(\frac{f(x_i)}{1-x_i} - \frac{f(x_j)}{1-x_j}\Bigr) .
\end{align}

\vspace{.4em}
\paragraph{\bf Step 3: Showing the three-point inequality is implied by two-point ones}
The last step is to use the inequality in the case of $K=2$ \eqref{eq:n=2-GDC} to prove \eqref{eq:3-point-inequality}. We have the following two technical lemmas whose proofs are in \cref{proof:FR-gradient-dominant-condition-sn}.
\begin{lemma}
\label{cond-b-GDC}
Given a convex function $f \in C^2(0,+\infty)$ with $f(1) = 0$ and $f'(1) = 0$, assume that \cref{cond:n=2-GDC} holds. Then there exists $0 < \delta < 1$, and $0 < \alpha_\delta \leq \alpha $, such that
\begin{align}
\label{eq:condition-b-GDC-3}
\bigl(f'(y) - f'(x)\bigr)^2  \geq  \alpha_\delta(\frac{1}{x} - \frac{1}{y})\Bigl( \frac{f(x)}{1 - x} - \frac{f(y)}{1 - y}\Bigr),\ \forall\ x \in (1-\delta, 1+\delta)\ \textrm{and}\ y > 0.
\end{align}
\end{lemma}
\begin{lemma}
\label{lem-2-point-to-3-point}
Given a convex function $f \in C^2(0,+\infty)$ with $f(1) = 0$ and $f'(1) = 0$ and under the constraint \eqref{eq:GDC-Df-discretized-constraints}, conditions \eqref{eq:n=2-GDC} and \eqref{eq:condition-b-GDC-3} implies \eqref{eq:3-point-inequality}.
\end{lemma}
With the above lemmas, the sufficient condition part is proved.
\end{proof}

\subsection{An easier-to-verify sufficient condition}
We also provide a simpler and easier-to-verify sufficient condition on $f$ that guarantees gradient dominance~\cref{cond:gradient-dominant-condition-f} holds.
\begin{proposition}
\label{theorem:FR-gradient-dominant-condition-suff}
    Assume  $f\in C^2(0,+\infty)$ is a  convex function with $f(1)=0$ and there exists $\alpha_s>0$ such that $x^2f''(x) > \alpha_s$ for $0<x\leq 1$.
    %and $f'(1)>0$, and $xf'(x)$ is concave.
    Then \cref{cond:n=2-GDC} holds. Note that the assumptions are satisfied with $f''(x) = x^p$ and $p\leq -2$.
\end{proposition}
\begin{proof}
The assumption $x^2f''(x) > \alpha_s$ for $0<x\leq 1$ implies that $f''(1) > \alpha_s > 0$. Combining this with the condition $f\in C^2$ ensures the existence of $0< \delta < \frac{1}{2}$, such that
\begin{equation}\label{e:f''bound}
    \frac{f''(1)}{2} \leq f''(x) \quad \text{for} \quad 1 - 2\delta \leq x\leq 1+2\delta.
\end{equation}
The convexity of $f(x)$ leads to
\begin{equation}
\label{eq:GDC-fx/1-x}
    (x - 1)f'(x) \geq f(x)-f(1)=f(x), \qquad \forall x > 0.
\end{equation}
With these preparations, we prove that for $0 < x < 1 < y$, \cref{cond:n=2-GDC} holds with $\alpha = \alpha_s \delta$.
This can be derived by multiplying the following two inequalities:
\begin{align*}
  &\phantom{{}={}}f'(y) - f'(x) =   \int_x^1 f''(\xi) {\rm d}\xi + \int_1^y f''(\xi) {\rm d}\xi
  \\
  &\geq
  \int_x^1 \frac{\alpha_s}{\xi^2} {\rm d}\xi + \min\{y-1,2\delta\} \frac{f''(1)}{2} \quad (\textrm{using } \xi^2f''(\xi)\geq\alpha_s \textrm { and \eqref{e:f''bound}})
  \nonumber
  \\
  &\geq
 \alpha_s\bigl(\frac{1}{x} - 1\bigr) +\delta f''(1)\bigl(1 - \frac{1}{y}\bigr)  \quad \!\!(\textrm{using } \min\{\frac{y-1}{2},\delta\} \geq\frac{\delta}{y}(y-1) \textrm{ due to } y>1>2\delta)
  \nonumber
  \\
  &\geq
   \alpha_s \delta \bigl(\frac{1}{x} - \frac{1}{y}\bigr) \quad (\textrm{using } \alpha_s > \alpha_s\delta \,\textrm{ and }\, f''(1)\geq \alpha_s) \nonumber
  \\
  &\phantom{{}={}} f'(y) -f'(x)  \geq  \frac{f(y)}{y - 1} + \frac{f(x)}{1 - x}  \quad (\textrm{using  \eqref{eq:GDC-fx/1-x} for both $x$ and $y$}).
\end{align*}
The proof is complete.
\end{proof}
\begin{remark}
    Recall that in general, the gradient dominance condition is implied by strong geodesic convexity of the energy functional. 
    In \cref{theorem:postive-uniform-convex}, a sufficient condition for strong geodesic convexity is that $xf'(x)$ is concave and $f''(1)>0$. If $xf'(x)$ is concave, we can show that $x^2f''(x)\geq f''(1)$ for any $0<x\leq 1$. To do so, we rewrite $x^2f''(x)$ in terms of $xf'(x)$ as
    $x^2f''(x)=x\bigl(xf'(x)\bigr)'-xf'(x).$ If $xf'(x)$ is concave, then $h(x) := \bigl(xf'(x)\bigr)'=xf''(x)+f'(x)$ is nonincreasing. We have
    \begin{align*}
        x^2f''(x)-f''(1)
        &=\bigl(xh(x)-xf'(x)\bigr)- \bigl(h(1)-f'(1)\bigr)\\
        &= \Bigl(\int_0^x h(x) {\rm d}\xi - \int_0^x h(\xi)\dd \xi\Bigr)  - \Bigl(\int_0^1 h(1){\rm d}\xi -  \int_0^1 h(\xi)\dd \xi\Bigr) \\
        &= \int_0^x h(x) {\rm d}\xi  + \int_x^1 h(\xi)\dd \xi - \int_0^1 h(1){\rm d}\xi\\
        &=\int_0^x h(x)- h(1)\dd \xi + \int_x^1 h(\xi)-h(1)\dd \xi  \\
        &\geq 0 \quad (\textrm{using } h(1) \leq h(x),h(\xi) \,\textrm{in}\, (0,1]).
    \end{align*}
    This implies that the assumption in \Cref{theorem:FR-gradient-dominant-condition-suff} is satisfied with $\alpha_s=f''(1)$, so the gradient dominance holds. The argument here thus establishes a connection between \Cref{theorem:FR-gradient-dominant-condition-suff} and \Cref{theorem:postive-uniform-convex}.
 \end{remark}

\section{Functional Inequality: Dual Gradient Dominance Condition}
\label{sec-Functional Inequality: Dual Gradient Dominance Condition}

The KL divergence is the most widely used energy functional for sampling. However, the negative results in \Cref{theorem:FR-displacement-convex,theorem:FR-gradient-dominant-condition} indicate that the concepts of geodesic convexity and gradient dominance condition are insufficient for analyzing the convergence of the Fisher-Rao gradient flow of the KL divergence.

In this section, we introduce a dual gradient dominance condition. This condition provides a more powerful functional inequality for understanding the convergence behavior of Fisher-Rao gradient flows associated with various $f$-divergences. To the best of our knowledge, such dual functional inequalities have not been explored in the Wasserstein setting, making this a novel discovery in the context of Fisher-Rao geometry.

\subsection{Non-increasing of general $f$-divergences}
Given a general convex function $\dualf\in C^2(0,+\infty), \dualf(1)=0$, we consider the following $\bar{f}$-divergence
\begin{equation}
D_\dualf[\rho \Vert  \rhopo]  = \int \rhopo \dualf(\frac{\rho}{\rhopo})\dd\theta.
\end{equation}
Along the Fisher-Rao gradient flow~\eqref{eq:FR-gf-f} with respect to  $D_f$, we have that
\begin{align}
\label{eq:dtDg}
&\frac{\dd}{\dd t} D_\dualf[\rho_t \Vert  \rhopo]
= \int \partial_t\rho_t \dualf'\bigl(\frac{\rho_t}{\rhopo}\bigr) \dd\theta
\\
%=& \int \Bigl(-\rho_t f'(\frac{\rho_t}{\rhopo}) + \rho_t \E_{\rho_t}\bigl[ f'(\frac{\rho_t}{\rhopo}) \bigr] \Bigr) \dualf'\bigl(\frac{\rho_t}{\rhopo}\bigr) \dd\theta \nonumber\\
=& \int \Bigl(-\rho_t(\theta)f'(\frac{\rho_t(\theta)}{\rhopo(\theta)}) + \rho_t(\theta) \int \rho_t(\theta') f'(\frac{\rho_t(\theta')}{\rhopo(\theta')}) \dd\theta' \Bigr) \dualf'\bigl(\frac{\rho_t(\theta)}{\rhopo(\theta)}\bigr) \dd\theta
\nonumber\\
=& \int \int \rho_t(\theta)\rho_t(\theta')\Bigl(-f'(\frac{\rho_t(\theta)}{\rhopo(\theta)}) +  f'(\frac{\rho_t(\theta')}{\rhopo(\theta')})  \Bigr) \dualf'\bigl(\frac{\rho_t(\theta)}{\rhopo(\theta)}\bigr) \dd\theta' \dd\theta
\nonumber\\
=& \frac{1}{2}\int \int \rho_t(\theta)\rho_t(\theta')\Bigl(-f'(\frac{\rho_t(\theta)}{\rhopo(\theta)}) +  f'(\frac{\rho_t(\theta')}{\rhopo(\theta')})  \Bigr) \Bigl(\dualf'\bigl(\frac{\rho_t(\theta)}{\rhopo(\theta)}\bigr) -  \dualf'\bigl(\frac{\rho_t(\theta')}{\rhopo(\theta')}\bigr) \Bigr) \dd\theta' \dd\theta
\nonumber\\
\leq& 0. \nonumber
\end{align}
Thus, any $D_\dualf$ is non-increasing along the Fisher-Rao gradient flow with respect to $D_f$. Here in the last inequality, we used the fact that both $f'(x)$ and $\dualf'(x)$ are non-decreasing functions.
\subsection{Dual gradient dominance condition}
Motivated by the non-increasing of $D_\dualf$ for any $\dualf$, in the following we introduce the dual gradient dominance condition. The idea is that we require the derivative in \eqref{eq:dtDg} to be lower bounded by a constant multiplying the sum of $D_f$ and $D_\dualf$, similar to the gradient dominance condition but applied to a different energy functional $D_\dualf$.
\begin{condition}[Dual gradient dominance condition]\label{cond:generalized-gradient-dominant-condition-f}
The function $f$ is said to satisfy the dual gradient dominance condition if there exists a constant $\alpha_f > 0$, depending only on $f$, such that
    \begin{equation}
    \begin{split}
-\int \Bigl(-\rho f'(\frac{\rho}{\rhopo}) + \rho \E_{\rho}\bigl[ f'(\frac{\rho}{\rhopo}) \bigr] \Bigr) \dualf'(\frac{\rho}{\rhopo})\dd\theta
\geq \alpha_f \Bigl( D_\dualf[\rho \Vert \rhopo]  + D_f[\rho \Vert \rhopo] \Bigr),
    \end{split}
\end{equation}
for any positive densities $\rho$ and $\rhopo$, provided that $f'(\frac{\rho}{\rhopo})$ is absolutely integrable with respect to $\rho$.
\end{condition}
\begin{newremark}
\label{rmk:dual-gdc}
  If the dual gradient dominance condition holds along the gradient flow of $D_f$, then we have
\begin{align*}
&\frac{{\rm d}}{{\rm d}t} D_\dualf[\rho_t \Vert  \rhopo] \leq -\alpha_f  D_\dualf[\rho_t \Vert  \rhopo], \\ 
&\frac{{\rm d}}{{\rm d}t}(D_f[\rho_t \Vert  \rhopo] + D_\dualf[\rho_t \Vert  \rhopo]) \leq -\alpha_f (D_f[\rho_t \Vert  \rhopo] + D_\dualf[\rho_t \Vert  \rhopo]), 
\end{align*}
which imply exponential convergence, i.e., $\bigO(\exp(-\alpha_f t))$, of $D_\dualf[\rho_t \Vert  \rhopo]$, provided the initial condition $\rho_0$ satisfies $D_\dualf[\rho_0 \Vert  \rhopo] < \infty$. Similar comments apply to $D_f[\rho_t \Vert  \rhopo] + D_\dualf[\rho_t \Vert  \rhopo]$ provided $D_f[\rho_0 \Vert  \rhopo] + D_\dualf[\rho_0 \Vert  \rhopo] < \infty$. We note that commonly used particle-based initializations in sampling often violate this condition, while this condition may be satisfied by employing Gaussian \cite{huang2022efficient,chen2023sampling}, Gaussian mixture \cite{lin2019fast,chen2024efficient,che2025stable}, or kernel density approximations \cite{maurais2024sampling,zhu2024kernel}. 
\end{newremark}
In the following, we will consider two special choices of $\dualf$:
\begin{enumerate}
    \item $\dualf(x) = \frac{(x-1)^2}{x}$, for which $D_\dualf$ corresponds to the reverse $\chi^2$ divergence.
    \item $\dualf(x) = xf(\frac{1}{x})$, which corresponds to the $*$-conjugate of $f$ \cite{osterreicher2002csiszar}.
\end{enumerate}
For both choices, we will prove the dual gradient dominance condition holds under certain assumptions on $f$. Specifically, the first choice leads to a general result, while the second choice leads to tighter results for specific $f$ including the KL divergence.
\subsection{General results using the reverse $\chi^2$ divergence as the dual}

We take $\dualf(x) = \frac{(x-1)^2}{x}$, the $f$-divergence associated with $\dualf$ is the reverse $\chi^2$ divergence
\begin{equation*}
D_\dualf[\rho\Vert\rhopo]=\chi^2[\rhopo \Vert  \rho]  = \int \frac{\rhopo^2}{\rho}\dd\theta - 1.
\end{equation*}
Along the Fisher-Rao gradient flow~\eqref{eq:FR-gf-f} with respect to $D_f$, the derivative \eqref{eq:dtDg} is given explicitly as
\begin{equation*}
\begin{aligned}
\frac{\dd}{\dd t} \chi^2[\rhopo \Vert  \rho_t]
&= \int \partial_t\rho_t \Bigl(1 - \bigl(\frac{\rhopo}{\rho_t}\bigr)^2\Bigr)  \dd\theta \\
&= \int \Bigl(-\rho_t f'(\frac{\rho_t}{\rhopo}) + \rho_t \E_{\rho_t}\bigl[ f'(\frac{\rho_t}{\rhopo}) \bigr] \Bigr) \Bigl(1 - \bigl(\frac{\rhopo}{\rho_t}\bigr)^2\Bigr) \dd\theta.
\end{aligned}
\end{equation*}
We have the following theorem:
\begin{theorem}
\label{theorem:generalized-ineq}
    Given a convex function $f \in C^2(0,+\infty)$ with $f(1) = 0$ and $f''(1) > 0$, there exists an $\alpha_{\chi^2,f} > 0$ that depends on $f$ only, such that
  \begin{subequations}
  \begin{align}
    -\int \Bigl(-\rho f'\Bigl(\frac{\rho}{\rhopo}\Bigr) + \rho \E_{\rho}\bigl[ f'\Bigl(\frac{\rho}{\rhopo}\Bigr) \bigr]\Bigr) \Bigl(1 - \bigl(\frac{\rhopo}{\rho}\bigr)^2\Bigr) \dd\theta
    \geq D_f[\rho \Vert \rhopo], \label{eq:generalized-gradient-dominant-condition-1}
    \\
    -\int \Bigl(-\rho f'(\frac{\rho}{\rhopo}) + \rho \E_{\rho}\bigl[ f'(\frac{\rho}{\rhopo}) \bigr] \Bigr) \Bigl(1 - \bigl(\frac{\rhopo}{\rho}\bigr)^2\Bigr) \dd\theta
 \geq \alpha_{\chi^2,f} \chi^2[\rhopo \Vert \rho] \label{eq:generalized-gradient-dominant-condition-2}.
\end{align}
\end{subequations}
The above inequalities hold for any {positive densities $\rho$ and $\rhopo$, provided that $f'(\frac{\rho}{\rhopo})$ is absolutely integrable with respect to $\rho$.
}
\end{theorem}

\begin{newremark}
    As a consequence of \Cref{theorem:generalized-ineq}, the Fisher-Rao gradient flow~\eqref{eq:FR-gf-f} associated with any $f$-divergence will satisfy the dual gradient dominance condition with rate $\alpha_f = \frac{\min\{1,\alpha_{\chi^2,f}\}}{2}$. Thus, assuming the flow is well-posed and using the discussions in \cref{rmk:dual-gdc}, we have that the flow converges exponentially fast:
 \begin{align*}
D_f[\rho_t \Vert  \rhopo] + \chi^2[\rhopo \Vert  \rho_t] \leq e^{-\frac{\min\{1,\alpha_{\chi^2,f}\}}{2} t}  \Bigl(D_f[\rho_0 \Vert  \rhopo] + \chi^2[\rhopo \Vert  \rho_0]\Bigr).
\end{align*}
We require that the initial condition $\rho_0$ satisfies $D_f[\rho_0 \Vert  \rhopo] + \chi^2[\rhopo \Vert  \rho_0] < \infty$.
\end{newremark}

\begin{proof}[Proof for \Cref{theorem:generalized-ineq}]
We denote
$\mu^2 = \int\frac{\rhopo^2}{\rho} \dd\theta \geq 1,$
then, the left hand side of \eqref{eq:generalized-gradient-dominant-condition-1} or \eqref{eq:generalized-gradient-dominant-condition-2} can be rewritten as
\begin{equation}
\label{eq:generalized-GDC-LHS}
\begin{split}
-\int \Bigl(-\rho f'\Bigl(\frac{\rho}{\rhopo}\Bigr) &+ \rho \E_{\rho}\bigl[ f'\Bigl(\frac{\rho}{\rhopo}\Bigr) \bigr]\Bigr) \Bigl(1 -  \frac{\rhopo^2}{\rho^2}\Bigr)\dd\theta\\
=&-\int f'\Bigl(\frac{\rho}{\rhopo}\Bigr)\frac{\rhopo^2}{\rho}\dd\theta + \int\frac{\rhopo^2}{\rho}\dd\theta \int{\rho}f'\Bigl(\frac{\rho}{\rhopo}\Bigr)\dd\theta\\
=&-\int f'\Bigl(\frac{\rho}{\rhopo}\Bigr)\Bigl(\frac{\rhopo^2}{\rho^2}-\mu^2\Bigr)\rho\dd\theta\\
=&\int \Bigl(f'\Bigl(\frac{1}{\mu}\Bigr)-f'\Bigl(\frac{\rho}{\rhopo}\Bigr)\Bigr)\Bigl(\frac{\rhopo^2}{\rho^2}-\mu^2\Bigr)\rho\dd\theta.
\end{split}
\end{equation}
To proceed, we use the following lemma of a local inequality; its proof can be found in \Cref{proof:dual-ineq-ns-cond}.
\begin{lemma}
\label{lem:generalized-gdc-1}
	Given smooth densities $\rho$ and $\rho^{*}$, define $\mu^2=\int\frac{\rhopo^2}{\rho}\dd\theta$, $\mu>0$. Then for convex $f$ and $\frac{\rho^*}{\rho} \in (0, +\infty)$, it holds that
 \begin{equation}
\label{eq:generalized-gdc-1}
\Bigl(f'(\frac{1}{\mu})-f'(\frac{\rho}{\rhopo})\Bigr)\Bigl(\frac{\rhopo^2}{\rho^2}-\mu^2\Bigr) \geq \mu(\frac{\rhopo}{\rho}-\mu)f'(\frac{1}{\mu})-\frac{\mu^2 \rhopo}{\rho}(f(\frac{1}{\mu})-f(\frac{\rho}{\rhopo})).
\end{equation}
\end{lemma}
Multiplying the local inequality~\eqref{eq:generalized-gdc-1} by $\rho$ and integrating over $\R^{d}$ leads to
\begin{align}
&\int \Bigl(f'\bigl(\frac{1}{\mu}\bigr)-f'\bigl(\frac{\rho}{\rhopo}\bigr)\Bigr)\Bigl(\frac{\rhopo^2}{\rho^2}-\mu^2\Bigr)\rho\dd\theta\nonumber\\
\geq&\int\mu(\rhopo-\mu\rho)f'\bigl(\frac{1}{\mu}\bigr)-\mu^2 \rhopo \Bigl(f(\frac{1}{\mu})-f(\frac{\rho}{\rhopo})\Bigr)\dd\theta \label{eq:generalized-gdc-1-intg}\\
=&\mu(1-\mu)f'(\frac{1}{\mu})-\mu^2f(\frac{1}{\mu})+\mu^2\int\rhopo f\bigl(\frac{\rho}{\rhopo}\bigr)\dd\theta
\geq  D_f[\rho \Vert \rhopo].   \nonumber
\end{align}
In the last inequality, we used the fact that $0=f(1) \geq f(\frac{1}{\mu})-(\frac{1}{\mu}-1)f'(\frac{1}{\mu})$ by the convexity of $f$ and that $\mu \geq 1$. This completes the proof of \eqref{eq:generalized-gradient-dominant-condition-1}.

To show \eqref{eq:generalized-gradient-dominant-condition-2}, we introduce another local inequality with its proof is in \Cref{proof:dual-ineq-ns-cond}:
\begin{lemma}
\label{lem:generalized-f''(1)}
    Given a convex function $f\in C^2(0,+\infty)$ with $f(1) = 0$ and $f''(1) > 0$,
     there exists $0 < \delta_f < \frac{1}{4}$, such that
    \begin{equation}
\label{eq:1-delta-neighbor-2}
    \frac{f''(1)}{2} \leq f''(x), \quad \text{for any } x \in (\frac{1}{1+2\delta_f}, \frac{1}{1-2\delta_f}).
\end{equation}
Moreover, there exists $\alpha_{\chi^2,f}' > 0$, such that
    \begin{equation}
        \label{eq:generalized-f''(1)}
\bigl(f'(\frac{1}{\mu}) -  f'(\frac{1}{x})\bigr)(x^2 - \mu^2)  \geq \alpha_{\chi^2,f}' (x - \mu)^2
\end{equation}
for any $x > 0$ and $\mu \in [1, 1+\delta_f]$.
\end{lemma}
We prove the dual gradient dominance condition \eqref{eq:generalized-gradient-dominant-condition-2} separately for two different scenarios, depending on whether $\mu$ is greater than $1+\delta_f$ or not.

When $\mu \leq 1+\delta_f$, taking $x =  \rhopo / \rho$ in \Cref{lem:generalized-f''(1)} leads to
\begin{equation}
\label{eq:generalized-gradient-dominant-condition-2a0}
\bigl(f'(\frac{1}{\mu}) -  f'(\frac{\rho}{\rhopo})\bigr)\bigl(\frac{\rhopo^2}{\rho^2} - \mu^2\bigr)  \geq \alpha_{\chi^2,f}' \bigl(\frac{\rhopo}{\rho} - \mu\bigr)^2.
\end{equation}
Multiplying the local inequality~\eqref{eq:generalized-gradient-dominant-condition-2a0} by $\rho$ and integrating over $\R^{d}$ leads to
\begin{equation}
\label{eq:generalized-gradient-dominant-condition-2a}
    \begin{split}
        &\int \Bigl(f'\bigl(\frac{1}{\mu}\bigr)-f'\bigl(\frac{\rho}{\rhopo}\bigr)\Bigr)\bigl(\frac{\rhopo^2}{\rho^2}-\mu^2\bigr)\rho\dd\theta
        \geq\alpha_{\chi^2,f}'\int  (\frac{\rhopo}{\rho} - \mu)^2 \rho \dd \theta\\
=&\alpha_{\chi^2,f}'\int \frac{\rhopo^2}{\rho} + \rho\mu^2 - 2 \rhopo \mu \dd \theta = 2\alpha_{\chi^2,f}'\mu(\mu-1) \quad (\textrm{using } \mu \geq  1)\\
\geq&\alpha_{\chi^2,f}'(\mu+1)(\mu-1) = \alpha_{\chi^2,f}'\chi^2[\rhopo \Vert  \rho].\\
    \end{split}
\end{equation}
When $\mu \geq 1+\delta_f$, utilizing \eqref{eq:generalized-gdc-1-intg}
yields
\begin{equation}
\begin{split}
 \label{eq:generalized-gradient-dominant-condition-2b0}
&\int \Bigl(f'\bigl(\frac{1}{\mu}\bigr)-f'\bigl(\frac{\rho}{\rhopo}\bigr)\Bigr)\bigl(\frac{\rhopo^2}{\rho^2}-\mu^2\bigr)\rho\dd\theta \\
%\geq&\int\bigl[\mu(\rhopo-\mu\rho)f'(\frac{1}{\mu})-\mu^2 \rhopo [f(\frac{1}{\mu})-f(\frac{\rho}{\rhopo})]\bigr]\dd\theta \\
\geq&\mu(1-\mu)f'(\frac{1}{\mu})-\mu^2f(\frac{1}{\mu})+\mu^2\int\rhopo f\Bigl(\frac{\rho}{\rhopo}\Bigr)\dd\theta  \\
\geq & \mu(1-\mu)f'(\frac{1}{\mu})-\mu^2f(\frac{1}{\mu}) .
\end{split}
\end{equation}
where in the last inequality, we used the fact that the $f$-divergence is non-negative.
Let us denote $h(x) = (x-1)f'(x) - f(x)$. Then for $x\in(0,1]$, $h(x)$ is nonnegative and non-increasing, since $h'(x) = (x-1)f''(x) \leq 0$ and $h(1)=0$. Therefore, the inequality~\eqref{eq:generalized-gradient-dominant-condition-2b0} can be further carried on to obtain
\begin{equation}
\begin{split}
\label{eq:generalized-gradient-dominant-condition-2b}
&\int \Bigl(f'\bigl(\frac{1}{\mu}\bigr)-f'\bigl(\frac{\rho}{\rhopo}\bigr)\Bigr)\Bigl(\frac{\rhopo^2}{\rho^2}-\mu^2\Bigr)\rho\dd\theta
\\
\geq & \mu(1-\mu)f'(\frac{1}{\mu})-\mu^2f(\frac{1}{\mu}) = \mu^2 h(\frac{1}{\mu})  \quad (\mu \geq 1+\delta_f
\textrm{ and } h \textrm{ non-increasing}, h\geq 0)
\\
\geq&  h\bigl(\frac{1}{1 + \delta_f}\bigr)(\mu^2-1) = h(\frac{1}{1 + \delta_f})\chi^2[\rhopo \Vert  \rho] .
\end{split}
\end{equation}
According to the definition of $\delta_f$ in \eqref{eq:generalized-f''(1)}, we have
\begin{equation}
   h\bigl(\frac{1}{1 + \delta_f}\bigr)
   = \int^{1}_{\frac{1}{1+\delta_f}} f''(x)(1-x) \dd x
   \geq \frac{f''(1)}{2}\int^{1}_{\frac{1}{1+\delta_f}} (1-x) \dd x = \frac{\delta_f^2 f''(1)}{4(1+\delta_f)^2} .
\end{equation}
In summary, combining the two scenarios, namely $\mu \leq 1+\delta_f$ in \eqref{eq:generalized-gradient-dominant-condition-2a} and $\mu \geq 1+\delta_f$ in \eqref{eq:generalized-gradient-dominant-condition-2b}, we conclude that \eqref{eq:generalized-gradient-dominant-condition-2} holds for
$$\alpha_{\chi^2,f} = \min\Bigl\{\alpha_{\chi^2,f}', \frac{\delta_f^2 f''(1)}{4(1+\delta_f)^2}  \Bigr\} > 0.$$
This completes the proof.
\end{proof}

\subsection{Tighter results using the $*$-conjugate $f$-divergence as the dual}
In the last subsection, we obtain general convergence results by choosing $\bar{f}$ as the reverse $\chi^2$ divergence. The constant $\alpha_f$ depends on $f$ only. However, it depends on some implicit constants such as $\delta_f, \alpha_{\chi^2,f}$ that are dependent on $f$. The goal of this subsection is to improve the estimate for specific cases and obtain more explicit constants.

We take $\dualf(x) = x f\bigl(\frac{1}{x}\bigr)$, known as as the $*$-conjugate function of $f$ \cite{osterreicher2002csiszar}. For this choice, the $\bar{f}$-divergence has the following dual form
\begin{equation}
\label{eq:g-divergence}
D_{\dualf}[\rho \Vert  \rhopo]
=\int\rhopo \dualf\Bigl(\frac{ \rho}{ \rhopo}\Bigr)\,\dd\theta
=\int\rho f\Bigl(\frac{ \rhopo}{ \rho}\Bigr)\,\dd\theta.
\end{equation}
We note that $\dualf(x) = x f\bigl(\frac{1}{x}\bigr)$ is also convex with $\dualf(1) = 0$.

Along the Fisher-Rao gradient flow~\eqref{eq:FR-gf-f} of $D_f$, we have
\begin{equation*}
\begin{aligned}
\frac{\dd}{\dd t} D_\dualf[\rho_t \Vert \rhopo]
&= \int \partial_t\rho_t\Bigl(f(\frac{\rhopo}{\rho_t}) -  \frac{\rhopo} {\rho_t}f'(\frac{\rhopo}{\rho_t})\Bigr) \dd\theta \\
&= \int \Bigl(-\rho_t f'(\frac{\rho_t}{\rhopo}) + \rho_t \E_{\rho_t}\bigl[ f'(\frac{\rho_t}{\rhopo}) \bigr] \Bigr) \Bigl(f(\frac{\rhopo}{\rho_t}) -  \frac{\rhopo}{\rho_t}f'(\frac{\rhopo}{\rho_t})\Bigr) \dd\theta.
\end{aligned}
\end{equation*}
The following theorem shows that the dual gradient dominance condition holds with an explicit constant when $f''$ is a polynomial.
\begin{theorem}
\label{theorem:dual-ineq}
Given any convex function $f$ specified by $f''(x) = x^{p}$ with $f(1) = 0$, when $p \leq -2$ or $-1 \leq p$, we have
\begin{equation}
\label{eq:dual-ineq-1}
\begin{aligned}
-\int \Bigl(-\rho f'(\frac{\rho}{\rhopo}) + \rho \E_{\rho}\bigl[ f'(\frac{\rho}{\rhopo}) \bigr] \Bigr) \Bigl(f(\frac{\rhopo}{\rho}) -  \frac{\rhopo}{\rho}f'(\frac{\rhopo}{\rho})\Bigr) \dd\theta \geq D_\dualf[\rho \Vert \rhopo] + D_f[\rho \Vert \rhopo],
\end{aligned}
\end{equation}
and when $-2 < p < -1$, we have 
\begin{equation}
\label{eq:dual-ineq-2}
\begin{aligned}
-\int \Bigl(-\rho f'(\frac{\rho}{\rhopo}) + \rho \E_{\rho}\bigl[ f'(\frac{\rho}{\rhopo}) \bigr] \Bigr) \Bigl(f(\frac{\rhopo}{\rho}) -  \frac{\rhopo}{\rho}f'(\frac{\rhopo}{\rho})\Bigr) \dd\theta \geq \frac{1}{2}\Bigl(D_\dualf[\rho \Vert \rhopo]+ D_f[\rho \Vert \rhopo]\Bigr).
\end{aligned}
\end{equation}
{The above inequalities hold for any positive densities $\rho$ and $\rhopo$, provided that $f'(\frac{\rho}{\rhopo})$ is absolutely integrable with respect to $\rho$.}
\end{theorem}
\begin{newremark}
    As a consequence of \Cref{theorem:dual-ineq}, we get a convergence result for the Fisher-Rao gradient flow associated with the KL-divergence. More precisely, for such gradient flow, if we choose an initial
density $\rho_0$ satisfying that ${\rm KL}[\rho_0 \Vert  \rhopo] + {\rm KL}[\rhopo\Vert\rho_0] < \infty$, then the Fisher-Rao gradient flow (assuming it is well-posed) converges exponentially fast with
\begin{align*}
{\rm KL}[\rho_t \Vert  \rhopo] + {\rm KL}[\rhopo \Vert  \rho_t] \leq e^{-t}  \Bigl({\rm KL}[\rho_0 \Vert  \rhopo] + {\rm KL}[\rhopo \Vert  \rho_0]\Bigr).
\end{align*}
\end{newremark}
% As a corollary of \cref{theorem:dual-ineq}, we deduce a sharp convergence estimate for the Fisher-Rao gradient flow~\eqref{eq:FR-gf-f} associated with the KL-divergence, as stated below.
% \begin{corollary}
%     For the Fisher-Rao gradient flow of  the KL-divergence
%     \begin{equation}
%     \frac{\partial \rho_t}{\partial t} = -\rho_t \Bigl(
% \bigl( \log \rho_t - \log \rhopo\bigr) - \E_{\rho_t}[ \log \rho_t -
% \log \rhopo] \Bigr),
% \end{equation} when we choose initial
% density $\rho_0$ satisfying that
% \begin{align*}
%     {\rm KL}[\rho_0 \Vert  \rhopo] + {\rm KL}[\rhopo\Vert\rho_0] < \infty,
% \end{align*}
% the Fisher-Rao gradient flow  converges exponentially fast with
% \begin{align*}
% {\rm KL}[\rho_t \Vert  \rhopo] + {\rm KL}[\rhopo \Vert  \rho_t] \leq e^{-t}  \Bigl({\rm KL}[\rho_0 \Vert  \rhopo] + {\rm KL}[\rhopo \Vert  \rho_0]\Bigr).
% \end{align*}
% \end{corollary}

\begin{proof}[Proof of \Cref{theorem:dual-ineq}]
First, we observe that both sides of these dual gradient dominance inequalities \eqref{eq:dual-ineq-1} and \eqref{eq:dual-ineq-2} remain the same under the transformation $f(x) \rightarrow f(x) + c(x - 1)$ for any $c$. Therefore, for each $\rho$, it suffices to prove that \eqref{eq:dual-ineq-1} and \eqref{eq:dual-ineq-2} hold for one representative element under the transformation.

When $p = -1$, we consider $f(x) = x\log x$; the corresponding $f$-divergence is the KL-divergence. We have the following
\begin{equation}
\begin{aligned}
&-\int \Bigl(-\rho f'(\frac{\rho}{\rhopo}) + \rho \E_{\rho}\bigl[ f'(\frac{\rho}{\rhopo}) \bigr] \Bigr) \Bigl(f(\frac{\rhopo}{\rho}) -  \frac{\rhopo}{\rho}f'(\frac{\rhopo}{\rho})\Bigr) \dd\theta
\\
=& \int \Bigl(-\rho \log(\frac{\rho}{\rhopo}) + \rho \E_{\rho}\bigl[ \log(\frac{\rho}{\rhopo}) \bigr] \Bigr) \frac{\rhopo}{\rho} \dd\theta
\\
=& \E_{\rhopo}\bigl[ \log(\frac{\rhopo}{\rho}) \bigr] + \E_{\rho}\bigl[ \log(\frac{\rho}{\rhopo}) \bigr] = D_\dualf[\rho \Vert \rhopo] + D_f[\rho \Vert \rhopo].
\end{aligned}
\end{equation}
Hence, \eqref{eq:dual-ineq-1} holds.

When $p = -2$, we consider $f(x) = -\log x$; the corresponding $f$-divergence is the reverse KL-divergence. We have
\begin{equation}
\begin{aligned}
&-\int \Bigl(-\rho f'(\frac{\rho}{\rhopo}) + \rho \E_{\rho}\bigl[ f'(\frac{\rho}{\rhopo}) \bigr] \Bigr) \Bigl(f(\frac{\rhopo}{\rho}) -  \frac{\rhopo}{\rho}f'(\frac{\rhopo}{\rho})\Bigr) \dd\theta
\\
=& -\int \bigl(\rhopo - \rho  \bigr)  \bigl(1 - \log\frac{\rhopo}{\rho} \bigr)\dd\theta
\\
=& \E_{\rhopo}\bigl[ \log(\frac{\rhopo}{\rho}) \bigr] - \E_{\rho}\bigl[ \log(\frac{\rhopo}{\rho}) \bigr] = D_\dualf[\rho \Vert \rhopo] + D_f[\rho \Vert \rhopo].
\end{aligned}
\end{equation}
Hence, \eqref{eq:dual-ineq-1} holds.

When $p \neq -1$ and $p \neq -2$, we consider $f(x) = \frac{x^{p+2} - 1}{(p+2)(p+1)}$. The corresponding $f$-divergences include $\chi^2$ divergence (when $p=0$). We have
\begin{equation}
\begin{aligned}
&-\int \Bigl(-\rho f'(\frac{\rho}{\rhopo}) + \rho \E_{\rho}\bigl[ f'(\frac{\rho}{\rhopo}) \bigr] \Bigr) \Bigl(f(\frac{\rhopo}{\rho}) -  \frac{\rhopo}{\rho}f'(\frac{\rhopo}{\rho})\Bigr) \dd\theta
\\
=& \frac{1}{(p+1)(p+2)}\int \Bigl(-\frac{\rho^{p+2}}{\rhopo^{p+1}}  +\rho \int \frac{\rho^{p+2}}{\rhopo^{p+1}} \dd\theta \Bigr)  \Bigl(\frac{1}{p+1} + \bigl(\frac{\rhopo}{\rho}\bigr)^{p+2}\Bigr)\dd\theta
\\
=& \frac{1}{(p+2)(p+1)} \Bigl(\int  \frac{\rho^{p+2}}{\rhopo^{p+1}} \dd\theta \int \frac{\rhopo^{p+2}}{\rho^{p+1}}  \dd\theta   - 1\Bigr).
\end{aligned}
\end{equation}
When $p<-2$ or $-1<p$, applying Jensen's inequality to the convex function $x^{p+2} ~(x > 0)$, we have
\begin{align*}
   \int  \frac{\rho^{p+2}}{\rhopo^{p+1}} \dd\theta = \E_{\rhopo}\Bigl[\bigl(\frac{\rho}{\rhopo}\bigr)^{p+2}\Bigr] \geq 1 \quad \textrm{ and } \quad \int  \frac{\rhopo^{p+2}}{\rho^{p+1}} \dd\theta = \E_{\rho}\Bigl[\bigl(\frac{\rhopo}{\rho}\bigr)^{p+2}\Bigr] \geq 1.
\end{align*}
Therefore, \eqref{eq:dual-ineq-1} holds using the following estimate:
\begin{equation}
\begin{aligned}
\frac{1}{(p+2)(p+1)} &\Bigl( \int  \frac{\rho^{p+2}}{\rhopo^{p+1}} \dd\theta \int \frac{\rhopo^{p+2}}{\rho^{p+1}}  \dd\theta  - 1\Bigr)
\\
&\geq \frac{1}{(p+2)(p+1)} \Bigl(  \int  \frac{\rho^{p+2}}{\rhopo^{p+1}} \dd\theta   - 1+ \int \frac{\rhopo^{p+2}}{\rho^{p+1}}  \dd\theta - 1\Bigr)
\\
&= D_\dualf[\rho \Vert \rhopo] + D_f[\rho \Vert \rhopo].
\end{aligned}
\end{equation}
When $-2 < p < -1$, applying Jensen's inequality to the concave function $x^{p+2} ~(x > 0)$, we have
$$\int  \frac{\rho^{p+2}}{\rhopo^{p+1}} \dd\theta  = \E_{\rhopo}\Bigl[\bigl(\frac{\rho}{\rhopo}\bigr)^{p+2}\Bigr] \leq 1 \quad \textrm{ and } \quad \int  \frac{\rhopo^{p+2}}{\rho^{p+1}} \dd\theta  = \E_{\rho}\Bigl[\bigl(\frac{\rhopo}{\rho}\bigr)^{p+2}\Bigr] \leq 1. $$ Note that $(p+1)(p+2) < 0$. We get \eqref{eq:dual-ineq-2} using the following inequality
\begin{equation}
\begin{aligned}
\frac{1}{(p+2)(p+1)} &\Bigl( \int  \frac{\rho^{p+2}}{\rhopo^{p+1}} \dd\theta \int \frac{\rhopo^{p+2}}{\rho^{p+1}}  \dd\theta  - 1\Bigr)
\\
&\geq \frac{1}{(p+2)(p+1)} \Bigl(  \frac{1}{2}\int  \frac{\rho^{p+2}}{\rhopo^{p+1}} \dd\theta  + \frac{1}{2}\int \frac{\rhopo^{p+2}}{\rho^{p+1}}  \dd\theta - 1\Bigr)
\\
&= \frac{1}{2}\Bigl(D_\dualf[\rho \Vert \rhopo] + D_f[\rho \Vert \rhopo]\Bigr).
\end{aligned}
\end{equation}
This concludes the proof.
\end{proof}

\section{Conclusion}
\label{sec-Conclusion}
In this study, we have investigated the geodesic convexity and functional inequalities related to Fisher-Rao gradient flows associated with various $f$-divergences.
We have constructed counterexamples to demonstrate that, in the Fisher-Rao geometry, $f$-divergences generally lack geodesic convexity or fail to satisfy the gradient dominance condition. A notable example is the KL divergence, which is widely used for sampling applications. Nevertheless, we have also developed sufficient conditions on $f$ under which $f$-divergences possess these properties.

Furthermore, we have introduced the dual gradient dominance condition and demonstrated its validity for a broad class of $f$-divergences, notably including the KL divergence. The constants in this condition depend solely on $f$, independent of the target distribution $\rhopo$. This characteristic sharply contrasts with the log-Sobolev inequality developed in the Wasserstein geometry, where constants depend on the target distribution.

As a direct consequence, assuming the well-posedness of the Fisher-Rao gradient flow, our results indicate that the flow exhibits exponential convergence, with a convergence rate independent of the target density.

Future research directions include establishing the well-posedness of the Fisher-Rao gradient flows and exploring the properties of more general, including non-smooth, energy functionals within the Fisher-Rao geometry as well as the related Wasserstein-Fisher-Rao geometry. Designing efficient numerical approximations for these Fisher-Rao gradient flows, particularly for sampling applications as discussed in \cite{lu2019accelerating,lu2022birth,huang2022efficient,lindsey2022ensemble,chen2024efficient,maurais2024sampling}, remains a promising and open area of research.

\appendix

\section{Lemmas for Proof of \Cref{theorem:FR-gradient-dominant-condition-sn}}
\label{proof:FR-gradient-dominant-condition-sn}
In this section, we prove \Cref{lemma-localization-of-density}, \Cref{cond-b-GDC}, and \Cref{lem-2-point-to-3-point},   which are needed to prove \Cref{theorem:FR-gradient-dominant-condition-sn}.

\begin{proof}[Proof of \Cref{lemma-localization-of-density}]
    We prove that there is a sequence of bounded domains $\{\Omega_n\}$ with $\cup_{n=1}^{\infty} \Omega_n = \R^d$, such that $\int_{\Omega_n}\rho \dd\theta = \int_{\Omega_n}\rhopo \dd\theta$. 
    
    Due to the smoothness of $\rho$ and $\rhopo$, when they are not the same, there exist $\theta_a$, $\theta_b$, $\delta > 0$ and $\epsilon > 0$, such that
\begin{align*}
 \rho(\theta) - \rhopo(\theta) > \epsilon, \qquad \forall \ \theta \in B_{\delta}(\theta_a),
 \\
 \rhopo(\theta) - \rho(\theta) > \epsilon, \qquad \forall\  \theta \in B_{\delta}(\theta_b),
\end{align*}
where $B_r(\theta)$ is the ball centered at $\theta$ with radius $r$. Let $|B_r(\theta)|$ denote the volume of the ball.

We first construct the bounded domain 
\begin{equation}
    \Omega_n = B_{R_n}(0) - \bigl(B_{\frac{1}{n}}(\theta_a) \cup B_{\frac{1}{n}}(\theta_b)\bigr)
\end{equation}
where $n > \frac{1}{\delta}$ and $R_n > n$ is large enough such that $B_{R_n}(0)$ contains both $B_{\delta}(\theta_a)$ and $B_{\delta}(\theta_b)$, $\int_{B_{R_n}^c(0)}\rho \dd\theta < \frac{|B_{1/n}(0)|\epsilon}{2}$, and $\int_{B_{R_n}^c(0)}\rhopo \dd\theta < \frac{|B_{1/n}(0)|\epsilon}{2}$. Without loss of generality, we assume $\int_{\Omega_n} \rhopo \dd\theta > \int_{\Omega_n} \rho \dd\theta$, then we define
\begin{equation}
    \Omega_n(r) = B_{R_n}(0) - \bigl(B_{r}(\theta_a) \cup B_{1/n}(\theta_b)\bigr)
\end{equation}
When $r = 0$, we have 
\begin{equation}
\begin{aligned}
    \int_{{\Omega}_n(r)} \rhopo \dd\theta &= 1-\int_{B_{R_n}(0)^c} \rhopo \dd\theta - \int_{B_{1/n}(\theta_b)} \rhopo \dd\theta\\
    & \leq 1 - \int_{B_{1/n}(\theta_b)} (\rho+\epsilon) \dd\theta \\
    & = 1-\int_{B_{1/n}(\theta_b)} \rho \dd\theta - \epsilon |B_{1/n}(0)|\\
    & < 1-\int_{B_{1/n}(\theta_b)} \rho \dd\theta - \int_{B_{R_n}(0)^c} \rho \dd\theta = \int_{{\Omega}_n(r)} \rho \dd\theta.
\end{aligned}
\end{equation}
Therefore, by continuity, there exists some $r_n \in (0, 1/n)$, such that  
$\int_{{\Omega}_n(r_n)} \rhopo \dd\theta = \int_{{\Omega}_n(r_n)} \rho \dd\theta$.
Using ${\Omega}_n(r_n)$ to replace $\Omega_n$ completes the proof. We can also make $\Omega_n$ monotonically increasing by using the union $\cup_{k=1}^n \Omega_k$ to replace $\Omega_n$.
\end{proof}

\bigskip
\begin{proof}[Proof of \Cref{cond-b-GDC}]
Assume that \Cref{cond:n=2-GDC} holds with $\alpha > 0$, namely
\begin{equation}
\label{appendix-A-7}
 \bigl(f'(y) - f'(x)\bigr)^2  \geq  \alpha(\frac{1}{x} - \frac{1}{y})\bigl( \frac{f(x)}{1 - x} - \frac{f(y)}{1 - y}\bigr)
\end{equation}
holds for any $0 < x < 1 < y$.
Let $y$ in \eqref{appendix-A-7} approach $1$, we get
\begin{align}
\label{eq:Df-n=2-necc-1}
  f'(x)^2  \geq  \alpha \frac{f(x)}{x}  \quad \textrm{ for } \quad x > 0 .
\end{align}
Here we used the fact that $\lim_{x\rightarrow 1} \frac{f(x)}{1 - x} = \lim_{x\rightarrow 1} f'(x) = 0$.
Now, letting $x$ further approach $1$ leads to that
\begin{align}
\label{eq:Df-n=2-necc-2}
  f''(1)  \geq  \frac{\alpha}{2} .
\end{align}
Here we used the fact that $\lim_{x\rightarrow 1} \frac{xf'(x)^2}{f(x)} = \lim_{x\rightarrow 1} 2xf''(x) + f'(x) = 2f''(1)$.
Since $f''(1)>0$ and $f\in C^2$, there exists $0< \delta' < \frac{1}{2}$, such that when $1 - 2\delta' \leq x\leq 1+2\delta'$,
\begin{equation}
\label{eq:GDC-delta}
    \frac{f''(1)}{2} \leq f''(x) \leq 2f''(1).
\end{equation}
Moreover, $f'(1)=0$ implies that there exists $0 < \delta < \delta'$, such at
\begin{equation}
\label{eq:GDC-delta-2}
    -f'(1-\delta) \leq \frac{\delta'}{2}f''(1). 
\end{equation}
Also, the convexity of $f$ implies that $f'(x)$ is non-decreasing and $\frac{f(x)}{1 - x}$ is non-increasing.
\Cref{cond-b-GDC} requires that \begin{align}
\label{eq:Df-n=2}
   \exists \alpha_\delta > 0 \quad \textrm{s.t.} \quad
 \bigl(f'(y) - f'(x)\bigr)^2  \geq  \alpha_\delta(\frac{1}{x} - \frac{1}{y})\bigl(\frac{f(x)}{1 - x} - \frac{f(y)}{1 - y}\bigr)  
\end{align}
holds for all $x \in (1-\delta, 1+\delta)$ and $y > 0$. Instead we consider a relaxed setting where $x,y > 0$, and $\{x,y\}\cap(1-\delta, 1+\delta)\neq \emptyset$. Due to the symmetry of \cref{eq:Df-n=2} with respect to $x$ and $y$, we assume without loss of generality that $x<y$ (when $x=y$, \cref{eq:Df-n=2} holds with $0\geq 0$). Moreover, by \cref{cond:n=2-GDC}, it suffices to consider the case where $1\not\in[x,y]$. In the following,  we will prove that \cref{eq:Df-n=2}
holds under conditions $x< y$,  $\{x,y\}\cap(1-\delta,1+\delta)\neq \emptyset$, and $1\not\in[x,y]$. We will then use the fact that 
$0\leq\delta\leq\delta'$ to divide the analysis into three cases:
\begin{enumerate}[leftmargin=4em,label=\textbf{Case \arabic*}]
    \item \label{enum-lem4.5-C.1} {\hspace{-0.5em}\bf .} $1-2\delta' < x < y < 1 + 2\delta'$,
    \item \label{enum-lem4.5-C.2}{\hspace{-0.5em}\bf .} $1 < x <1+ \delta'  < 1+2\delta' < y$,
    \item \label{enum-lem4.5-C.3}{\hspace{-0.5em}\bf .} $0<x <1 - 2\delta'  < 1-\delta < y < 1$.
\end{enumerate}

For \ref{enum-lem4.5-C.1},  \eqref{eq:Df-n=2} can be obtained with $\alpha_{\delta_1} = \frac{f''(1)(1-2\delta')^2}{4}$ from multiplying the following two inequalities
\begin{align}
  f'(y) - f'(x)
  &=   \int_x^y f''(\xi) \dd\xi
  \geq \frac{f''(1)}{2}(y-x) \quad (\textrm{using \eqref{eq:GDC-delta}}) \label{eq:case-1.1}
  \\
  &\geq \frac{f''(1)(1-2\delta')^2}{2}\bigl(\frac{1}{x} - \frac{1}{y}\bigr), \nonumber
  \\
   \frac{f(x)}{1 - x} - \frac{f(y)}{1 - y}
   &= \int_x^y \frac{-f(\xi) + f'(\xi)(\xi-1)}{(1-\xi)^2}\dd\xi = \int_x^y \frac{\int_1^\xi f''(z)(z-1)\dd z}{(1-\xi)^2}\dd\xi \nonumber\\
   &\leq \int_x^y \frac{2f''(1)\bigl|\int_1^\xi (z-1)\dd z\bigr|}{(1-\xi)^2}\dd\xi \quad (\textrm{using \eqref{eq:GDC-delta}})
   \\
   &=  f''(1)(y-x) \nonumber
   \\
   &\leq 2\bigl(f'(y) - f'(x)\bigr) .\nonumber
\end{align}
Here in the last inequality, we used \eqref{eq:case-1.1}.

For \ref{enum-lem4.5-C.2},  \eqref{eq:Df-n=2} can be obtained with $\alpha_{\delta_2} = \frac{\delta' f''(1)}{6}$ from multiplying the following two inequalities together with \eqref{eq:GDC-delta} leading to
\begin{align}
  f'(y) - f'(x)
&=   \int_x^y f''(\xi) {\rm d}\xi \geq \int_{1+\delta'}^{1 + 2\delta'}f''(\xi) \dd\xi \geq \frac{\delta' f''(1)}{2}\geq  \frac{\delta'  f''(1)}{2} \bigl(\frac{1}{x} - \frac{1}{y}\bigr),\label{eq:case-2.1}
  \\
  \frac{f(x)}{1 - x} - \frac{f(y)}{1 - y} &= \int_1^y f''(\xi)\frac{y-\xi}{y-1}{\rm d}\xi - \int_1^x f''(\xi)\frac{x-\xi}{x-1}{\rm d}\xi
  \quad (\textrm{using $f(1)=f'(1)=0$})\nonumber\\
  &= \frac{y-x}{(y-1)(x-1)}\int_1^x f''(\xi)(\xi-1){\rm d}\xi + \int_x^y f''(\xi)\frac{y-\xi}{y - 1}{\rm d}\xi \\
   &\leq \frac{1}{x-1}\int_1^x 2f''(1)(\xi-1){\rm d}\xi + \int_x^y f''(\xi){\rm d}\xi \nonumber\\
&\leq  \delta' f''(1)  + f'(y) - f'(x) \nonumber
\\
&\leq 3\bigl(f'(y) - f'(x)\bigr) . \nonumber
\end{align}
Here in the last inequality, we used \eqref{eq:case-2.1}.

For \ref{enum-lem4.5-C.3},
\eqref{eq:Df-n=2} can be obtained with $\alpha_{\delta_3} = \frac{\alpha}{4}$. In fact, 
the right hand side of \eqref{eq:Df-n=2} can be bounded as
\begin{align}
\label{eq:lemm4-lhs}
   (\frac{1}{x} - \frac{1}{y})\bigl( \frac{f(x)}{1 - x} - \frac{f(y)}{1 - y}\bigr)   \leq \frac{f(x)}{x},
\end{align}
by taking $y=1$ on the left hand side, since both $\frac{1}{y}$ and $\frac{f(y)}{1 - y}$ are decreasing. For the left hand side of \cref{eq:Df-n=2}, both $f'(x)$ and $f'(y)$ are negative. Using \eqref{eq:GDC-delta} and \eqref{eq:GDC-delta-2} leads to
\begin{align}
\label{eq:f'x-f'y}
  -f'(x)
&=   \int_x^1 f''(\xi) {\rm d}\xi \geq \int_{1-2\delta'}^{1}f''(\xi) \dd\xi \geq \delta' f''(1) \geq -2f'(1 - \delta) \geq -2f'(y).
\end{align}
Using \cref{eq:f'x-f'y},\cref{eq:lemm4-lhs}, and \cref{eq:Df-n=2-necc-1} leads to
\begin{align*}
  \Bigl(f'(x) -f'(y) \Bigr)^2 \geq \frac{f'(x)^2}{4} \geq \frac{\alpha}{4}\frac{f(x)}{x}\geq \frac{\alpha}{4}(\frac{1}{x} - \frac{1}{y})\bigl( \frac{f(x)}{1 - x} - \frac{f(y)}{1 - y}\bigr) .
\end{align*}
Finally, choosing $\delta>0$ as in \cref{eq:GDC-delta,eq:GDC-delta-2} and combining all three cases, we conclude that \eqref{eq:Df-n=2} holds with $\alpha_\delta = \min\{\frac{f''(1)(1-2\delta')}{4},\frac{\delta'f''(1)}{6}, \frac{\alpha}{4}\}$. 
\end{proof}
\bigskip

\begin{proof}[Proof of \Cref{lem-2-point-to-3-point}]
In this lemma, we assume that there exists $0 < \delta < 1$ and $0 < \alpha_\delta \leq \alpha $ such that
\begin{align}
\label{appendix-A-1}
\bigl(f'(y) - f'(x)\bigr)^2  \geq  \alpha_\delta(\frac{1}{x} - \frac{1}{y})\bigl( \frac{f(x)}{1 - x} - \frac{f(y)}{1 - y}\bigr),
\end{align}
holds for all $0 < x < 1 < y$ (\Cref{cond:n=2-GDC}), and for any $x \in (1-\delta, 1+\delta)$ and $ y > 0$ (\Cref{cond-b-GDC}). We also have the constraint
\begin{align}
\label{appendix-A-2-constraint}
    &\sum_{i=1}^3 \rho_i = 1,  \quad \sum_{i=1}^3 \frac{\rho_i}{x_i} = 1, \quad
    \rho_i > 0, \quad x_i > 0.
\end{align}
Our goal is to show
\begin{align}
\label{appendix-3-point-inequality}
 \sum_{1\leq i,j\leq 3}  \rho_i\rho_j\Bigl(f'(x_i)-f'(x_j)\Bigr)^2 \geq \alpha_f\sum_{1\leq i,j\leq 3}  \rho_i\rho_j(\frac{1}{x_i} - \frac{1}{x_j}) \Bigl(\frac{f(x_i)}{1-x_i} - \frac{f(x_j)}{1-x_j}\Bigr) .
\end{align}

We assume that $x_1 \leq x_2 \leq x_3$, then due to the constraint \eqref{appendix-A-2-constraint}, we have either $x_1=x_2=x_3=1$ or $x_1 < 1 < x_3$. 

For the first scenario, \eqref{appendix-3-point-inequality} holds with any $\alpha_f$. For the second scenario, when $1 - \delta < x_2 < 1+ \delta$, the relationships between pairs $(x_1, x_2)$, $(x_1, x_3)$ and $(x_2, x_3)$ fall into the conditions described in \eqref{appendix-A-1}. Thus, \eqref{appendix-3-point-inequality} holds with $\alpha_f = \alpha_\delta$.
Now, we only need to verify the following cases:

\begin{enumerate}[leftmargin=4em,label=\textbf{Case \arabic*}]
    \item \label{enum-C.1} {\hspace{-0.5em}\bf .} $0 < x_1 < 1 < 1+\delta\leq x_2 \leq x_3$,
    \item \label{enum-C.2}{\hspace{-0.5em}\bf .} $0 < x_1 \leq x_2 \leq 1- \delta  < 1 < x_3$.
\end{enumerate}

For \ref{enum-C.1}, we will first show that
\begin{equation}
\label{app-eq:scenario-1-gdc}
      \rho_1\rho_3(\frac{1}{x_1}-\frac{1}{x_3}) \Bigl(\frac{f(x_1)}{1 - x_1} - \frac{f(x_3)}{1 - x_3}\Bigr)  \geq
      \delta \rho_2\rho_3(\frac{1}{x_2}-\frac{1}{x_3}) \Bigl(\frac{f(x_2)}{1 - x_2} - \frac{f(x_3)}{1 - x_3}\Bigr).
     \end{equation}
In fact, the convexity of $f(x)$ and the condition $f(1)=0$ imply that $\frac{f(x)}{1 - x}$ is decreasing. Consequently, $\frac{f(x_1)}{1 - x_1} - \frac{f(x_3)}{1 - x_3}  \geq \frac{f(x_2)}{1 - x_2} - \frac{f(x_3)}{1 - x_3} \geq 0 $, hence \cref{app-eq:scenario-1-gdc} can be obtained if we can show
      $\rho_1(\frac{1}{x_1}-\frac{1}{x_3}) \geq
      \delta \rho_2 (\frac{1}{x_2}-\frac{1}{x_3}).$
This can be achieved using the following estimate:
\begin{equation*}
      \rho_1(\frac{1}{x_1} - \frac{1}{x_3}) \geq \rho_1(\frac{1}{x_1} - 1)= \rho_2(1 - \frac{1}{x_2}) + \rho_3(1 - \frac{1}{x_3}) \geq \rho_2(1 - \frac{1}{x_2}) \geq
      \delta \rho_2(\frac{1}{x_2} - \frac{1}{x_3}).
     \end{equation*}
Here the equality follows from \cref{appendix-A-2-constraint}, and the last inequality uses the fact that $x_2\geq 1 + \delta$ and $x_3 > 0$.

With \eqref{app-eq:scenario-1-gdc}, we can prove that \eqref{appendix-3-point-inequality} holds with $\alpha_f = \frac{\alpha_\delta \delta}{1 + \delta}$. To do so, we denote
\begin{equation}
\label{app-eq:condition-b-GDC-F}
    F_{i,j} = \rho_i\rho_j(\frac{1}{x_i}-\frac{1}{x_j}) \Bigl(\frac{f(x_i)}{1 - x_i} - \frac{f(x_j)}{1 - x_j}\Bigr),
\end{equation}
then we have
\begin{equation*}
\begin{split}
    \sum_{1 \leq i< j\leq 3} \rho_i \rho_j \Bigl(f'(x_i) -  f'(x_j)\Bigr)^2
    &\geq \alpha_\delta (F_{3,1} + F_{2,1}) \quad (\textrm{using  \eqref{appendix-A-1}})
    \\
    &\geq \alpha_\delta (\frac{\delta}{1 + \delta}F_{3,1} + \frac{\delta}{1 + \delta}F_{3,2} + F_{2,1})
    \quad (\textrm{using \eqref{app-eq:scenario-1-gdc}})
    \\
    &\geq \frac{\alpha_\delta \delta}{1 + \delta} (F_{3,1} + F_{3,2} + F_{2,1}),
\end{split}
\end{equation*}
which shows that \eqref{appendix-3-point-inequality} holds with $\alpha_f = \frac{\alpha_\delta \delta}{1 + \delta}$.

For \ref{enum-C.2}, we will first show that
\begin{equation}
\label{app-eq:scenario-2-gdc}
      \rho_1\rho_3(\frac{1}{x_1}-\frac{1}{x_3}) \Bigl(\frac{f(x_1)}{1 - x_1} - \frac{f(x_3)}{1 - x_3}\Bigr)  \geq
      \delta \rho_2\rho_1(\frac{1}{x_1}-\frac{1}{x_2}) \Bigl(\frac{f(x_1)}{1 - x_1} - \frac{f(x_2)}{1 - x_2}\Bigr)  .
     \end{equation}
In fact, the convexity of $f(x)$ and the condition $f(1)=0$ imply that $\frac{f(x)}{1 - x}$ is decreasing. Consequently, $\frac{f(x_1)}{1 - x_1} - \frac{f(x_3)}{1 - x_3}  \geq \frac{f(x_2)}{1 - x_2} - \frac{f(x_3)}{1 - x_3} \geq 0$. Combining with $\frac{1}{x_1}-\frac{1}{x_3} \geq \frac{1}{x_1}-\frac{1}{x_2} \geq 0$, this implies that \eqref{app-eq:scenario-2-gdc} holds if $\rho_3  \geq
      \delta \rho_2$. This can be derived by the following estimate:
\begin{equation*}
      \rho_3\geq \rho_3(1 - \frac{1}{x_3}) = \rho_1(\frac{1}{x_1} - 1) + \rho_2(\frac{1}{x_2} - 1) \geq \rho_2(\frac{1}{x_2} - 1) \geq
      \delta \rho_2.
     \end{equation*}
Here, in the last inequality, we used the fact that $x_2 \leq 1 - \delta$.
Next, we will prove that \eqref{appendix-3-point-inequality} holds with $\alpha_f = \frac{\alpha_\delta \delta}{1 + \delta}$. Using the definition in \eqref{app-eq:condition-b-GDC-F}, we have
\begin{equation*}
\begin{split}
    \sum_{1 \leq i< j\leq 3} \rho_i \rho_j \Bigl(f'(x_i) -  f'(x_j)\Bigr)^2
    &\geq \alpha_\delta (F_{3,1} + F_{3,2}) \quad (\textrm{using  \eqref{appendix-A-1}})
    \\
    &\geq \alpha_\delta (\frac{\delta}{1 + \delta}F_{3,1} + \frac{\delta}{1 + \delta}F_{2,1} + F_{3,2})
    \quad (\textrm{using \eqref{app-eq:scenario-2-gdc}})
    \\
    &\geq \frac{\alpha_\delta \delta}{1 + \delta} (F_{3,1} + F_{3,2} + F_{2,1}).
\end{split}
\end{equation*}
Combining both \ref{enum-C.1} and \ref{enum-C.2}, we complete the proof.
\end{proof}

\section{Lemmas for Proof of \Cref{theorem:generalized-ineq}}
\label{proof:dual-ineq-ns-cond}
In this section, we provide the proofs of \Cref{lem:generalized-gdc-1} and \cref{lem:generalized-f''(1)}, which are needed in the proof of \Cref{theorem:generalized-ineq}.
\begin{proof}[Proof of \Cref{lem:generalized-gdc-1}]
The convexity of $f$ leads to that
\begin{align}\label{f1}
f(y)\geq f(x)+(y-x)f'(x),\quad \forall\ y>0,\ x>0.
\end{align}
Adding $ \bigl(f'(y)-f'(x)\bigr)(y-x)$ to both sides gives
\begin{align}\label{f2}
\bigl(f'(y)-f'(x)\bigr)(y-x)\geq (y-x)f'(y)-f(y)+f(x)\geq0,\quad \forall\ y>0,\ x>0.
\end{align}
Then, for $x>0$, using \cref{f2}, we obtain
\begin{align}
\bigl(f'(\frac{1}{\mu})-f'(\frac{1}{x})\bigr)(\frac{1}{\mu}-\frac{1}{x})  \geq & (\frac{1}{\mu}-\frac{1}{x})f'(\frac{1}{\mu})-f(\frac{1}{\mu})+f(\frac{1}{x})\geq0, \label{lem:generalized-gdc-1-1}
\\
\bigl(f'(\frac{1}{\mu})-f'(\frac{1}{x})\bigr)(x^2-\mu^2) =&\bigl(f'(\frac{1}{\mu})-f'(\frac{1}{x})\bigr)(\frac{1}{\mu}-\frac{1}{x})\mu x(\mu+x) \nonumber
\\ \geq&\bigl(f'(\frac{1}{\mu})-f'(\frac{1}{x})\bigr)(\frac{1}{\mu}-\frac{1}{x})\mu^2 x \nonumber
\\ \geq&\bigl((\frac{1}{\mu}-\frac{1}{x})f'(\frac{1}{\mu})-f(\frac{1}{\mu})+f(\frac{1}{x})\bigr)\mu^2 x \nonumber
\\ =&\mu(x-\mu)f'(\frac{1}{\mu})-\mu^2 x\bigl(f(\frac{1}{\mu})-f(\frac{1}{x})\bigr).\label{lem:generalized-gdc-1-2}
\end{align}
The first inequality follows from $\bigl(f'(\frac{1}{\mu})-f'(\frac{1}{x})\bigr)(\frac{1}{\mu}-\frac{1}{x}) \geq 0$, and the second from \cref{lem:generalized-gdc-1-1}.
Taking $x=\rhopo/\rho$ in \eqref{lem:generalized-gdc-1-2} leads to
\begin{align*}
\Bigl(f'(\frac{1}{\mu})-f'(\frac{\rho}{\rhopo})\Bigr) \Bigl(\frac{\rhopo^2}{\rho^2}-\mu^2\Bigr)\geq \mu(\frac{\rhopo}{\rho}-\mu)f'(\frac{1}{\mu})-\frac{\mu^2 \rhopo}{\rho}\Bigl(f(\frac{1}{\mu})-f(\frac{\rho}{\rhopo})\Bigr),
\end{align*}
which leads to the desired inequality. This completes the proof.
\end{proof}

\begin{proof}[Proof of \cref{lem:generalized-f''(1)}]
Since $f \in C^2$ and $f''(1) > 0$, there exists $0 < \delta_f < \frac{1}{4}$, such that for any $x \in (\frac{1}{1+2\delta_f}, \frac{1}{1-2\delta_f})$, we have $\frac{f''(1)}{2} \leq f''(x).$

For the left hand side of \eqref{eq:generalized-f''(1)}, for any $x,y$ satisfying $y>x>0$ and $\{x,y\}\cap[1,1+\delta_f]\neq\emptyset $, we have
\begin{align}
    \label{eq:f'(1/x)-f'(1/y)}&f'(\frac{1}{x}) - f'(\frac{1}{y}) = \int_{1/y}^{1/x}f''(\xi){\rm d}\xi
    \geq \frac{f''(1)}{2}\int_{(\frac{1}{y},\frac{1}{x})\cap(\frac{1}{1+2\delta_f}, \frac{1}{1-2\delta_f})}{\rm d}\xi
    \\
    \geq& \begin{cases}
       \frac{f''(1)}{2}\Bigl(\frac{1}{x} - \frac{1}{y}\Bigr) \geq \frac{f''(1)}{1+2\delta_f}\frac{y-x}{y+x} & 1-2\delta_f < x < y< 1+2\delta_f
       \\
       \frac{f''(1)}{2}\Bigl(\frac{1}{1-2\delta_f} - \frac{1}{y}\Bigr) \geq \frac{f''(1)\delta_f}{1-2\delta_f}\frac{y-x}{y+x} & 0<x\leq 1-2\delta_f <1\leq y\leq 1+\delta_f
       \\
       \frac{f''(1)}{2}\Bigl(\frac{1}{x} - \frac{1}{1 + 2\delta_f}\Bigr)\geq \frac{f''(1)\delta_f}{2(1+2\delta_f)(1+\delta_f)} \frac{y-x}{y+x} & 1\leq x\leq  1+\delta_f < 1+2\delta_f\leq y
    \end{cases}\nonumber
    \\
    \geq& \alpha_{\chi^2,f}' \frac{y-x}{y+x},\nonumber
\end{align}
where $\alpha_{\chi^2,f}' = \frac{f''(1)\delta_f}{2(1+2\delta_f)(1+\delta_f)} > 0$. These cases  cover all situations with $y>x>0,\ \{x,y\}\cap[1,1+\delta_f]\neq\emptyset$.

Then we consider \eqref{eq:generalized-f''(1)} with $\mu\in[1,1+\delta_f]$. When $0 < x < \mu$, by using \eqref{eq:f'(1/x)-f'(1/y)}, we have
\begin{equation}
\begin{split}
\bigl(f'(\frac{1}{\mu}) -  f'(\frac{1}{x})\bigr)(x^2 - \mu^2)  \geq \alpha_{\chi^2,f}'\frac{\frac{1}{x} - \frac{1}{\mu}}{\frac{1}{x} + \frac{1}{\mu}} (\mu^2 - x^2) = \alpha_{\chi^2,f}' (\mu - x)^2.
\end{split}
\end{equation}
When $\mu < x$, by using \eqref{eq:f'(1/x)-f'(1/y)}, we have
\begin{equation}
\begin{split}
\bigl(f'(\frac{1}{\mu}) -  f'(\frac{1}{x})\bigr)(x^2 - \mu^2)  \geq \alpha_{\chi^2,f}'\frac{\frac{1}{\mu} - \frac{1}{x}}{\frac{1}{x} + \frac{1}{\mu}} (x^2 - \mu^2) = \alpha_{\chi^2,f}' (\mu - x)^2.
\end{split}
\end{equation}
If $x=\mu$ then \eqref{eq:generalized-f''(1)} becomes $0\geq0$ and is still true. This completes the proof.
\end{proof}

\bibliographystyle{siamplain}
\bibliography{references}

\end{document}